\documentclass[11pt]{article}
\usepackage{amsmath,amssymb,amsthm}
\usepackage[margin=1in]{geometry}
\usepackage{color}
\usepackage{hyperref}

\numberwithin{equation}{section}
\theoremstyle{plain}                
\newtheorem{theorem}{Theorem}[section]
\newtheorem{lemma}[theorem]{Lemma}
\newtheorem{proposition}[theorem]{Proposition}
\newtheorem{corollary}[theorem]{Corollary}
\theoremstyle{definition}           
\newtheorem{definition}[theorem]{Definition}

\newtheorem{remark}[theorem]{Remark}

\providecommand{\R}{} \renewcommand{\R}{{\mathbb R}}

\newcommand{\N}{{\mathbb N}}

\renewcommand{\AA}{{\mathcal A}}

\newcommand{\tw}{{\widetilde{w}}}
\newcommand{\tv}{{\widetilde{v}}}

\newcommand{\eps}{\varepsilon}

\def\R{\mathbb{R}}

\def\E{\mathbb{E}}
\def\P{\mathbb{P}}

\def\R{\mathbb{R}}
\def\N{\mathbb{N}}

\def\P{\mathbb{P}}
\def\E{\mathbb{E}}

\def\rr{\mathbb{R}}

\def\lp{L^p_a(\rr)}
\def\wi{W^{1,\infty}_a(\rr)}
\def\li{L^\infty_a(\rr)}
\def\eps{\varepsilon}

\begin{document}

\title{Optimal Investment with Switching Preferences}
\author{Yu-Jui Huang\thanks{%
Department of Applied Mathematics, University of Colorado, Boulder, CO 80309-0526, USA, email: \texttt{yujui.huang@colorado.edu}
} \and Liviu Ignat\thanks{%
Institute of Mathematics ``Simion Stoilow'' of the Romanian Academy, 21 Calea Grivitei Street, 010702 Bucharest, Romania.
 National University of Science and Technology Politehnica Bucharest, 313 Splaiul Independen\c tei, 060042 Bucharest, Romania.
 Academy of Romanian Scientists, Ilfov Street, no. 3, Bucharest, Romania, email: \texttt{liviu.ignat@gmail.com}} 
 \and Traian A. Pirvu\thanks{%
Department of Mathematics and Statistics, McMaster University, 1280 Main Street West, Hamilton, ON, L8S 4K1, Canada, email: \texttt{pirvut@mcmaster.ca}} 
\and Reihaneh Vafadar\thanks{%
Department of Mathematics and Statistics, McMaster University, 1280 Main Street West, Hamilton, ON, L8S 4K1, Canada, email: \texttt{vafadar@mcmaster.ca}}
}
\maketitle

\begin{abstract}
Major life events can significantly increase individuals' risk aversion over a sustained period of time, as empirical studies reveal. How such an event-triggered shift of risk preferences impacts optimal investment is the focus of this paper. On a finite time horizon where a major life event may occur independently of the financial market, an investor aims to maximize her expected utility from terminal wealth while foreseeing a potential change in her risk aversion. We find that the associated Hamilton--Jacobi--Bellman (HJB) equation involves the post-event optimal value function (under elevated but fixed risk aversion after the event's occurrence), and the Fenchel--Legendre transform fails to linearize this HJB equation: it yields a parabolic equation with a fully nonlinear term, induced precisely by the post-event optimal value function. Through a combination of fixed-point, compactness, and verification arguments,  
we establish the existence of a positive convex classical solution with suitable growth to the fully-nonlinear parabolic equation. The convex conjugate of this solution is shown to satisfy the HJB equation and coincides with the pre-event optimal value function. The optimal trading strategy is obtained by concatenating the optimal pre-event and post-event strategies---the former is expressed in terms of the solution to the HJB equation and the latter is traditional Merton's ratio.
\end{abstract}


\begin{flushleft}
\end{flushleft}

\section{Introduction}\label{sec:intro}
Risk preferences may change due to major life events,  
such as parenthood and bereavement (Kettlewell \cite{Kettlewell19}; G\"{o}rlitz and Tamm \cite{GT20}; Browne et al.\ \cite{BJRS22}), health shocks (Decker and Schmitz \cite{DS16}), exposure to violent crime (Voors et al.\ \cite{Voors12}; Brown et al.\ \cite{BMTV19}), and experiencing natural disasters 
(Cameron and Shah \cite{CS15}; Cassar et al.\ \cite{CHvK17}; Hanaoka et al.\ \cite{HSW18}). Empirical findings in these studies indicate that many major life events tend to significantly increase individuals' risk aversion, and this more conservative stance  can last several years or much longer. 

This paper aims to elucidate how an event-triggered shift of risk preferences impacts optimal investment. Specifically, in a Black--Scholes market, we consider an investor who chooses a trading strategy to maximize her expected utility from terminal wealth over a fixed time horizon $T>0$, while a major life event may occur (independently of the financial market) before the terminal time $T$. We assume that the investor uses an isoelastic utility function $U(x) = \frac{x^{1-\gamma}}{1-\gamma}$, where $0<\gamma\neq 1$ represents her risk aversion, and the occurrence of a major life event will shift the level of risk aversion from $\gamma_1$ to $\gamma_2$ (with $\gamma_1<\gamma_2$). Our goal is to characterize the investor's optimal trading strategy under this potential change of risk aversion over the finite time horizon $T$. For tractability, we assume that the investor's risk aversion will stay at the elevated level $\gamma_2$ after the major life event, which is in fact consistent with the aforementioned empirical studies over a short to intermediate investment horizon $T$. 

In the standard case of fixed risk aversion, it is well-known that the Fenchel--Legendre transform {\it linearizes} the Hamilton--Jacobi--Bellman (HJB) equation associated with the optimal investment problem, which significantly facilitates the characterization of an optimal trading strategy. This simplification does not fully work in our case. After the occurrence of a major life event, as the investor has a fixed level of risk aversion $\gamma_2$, the standard linearization works and it gives explicit formulas for the {\it post-event} optimal value function, denoted by $v(t,x,2)$, and trading strategy. Before the occurrence of the major life event, since the investor foresees the possible change of her risk aversion from $\gamma_1$ to $\gamma_2$, the HJB equation for the {\it pre-event} optimal value function, denoted by $v(t,x,1)$, depends on $v(t,x,2)$. While one can apply the Fenchel--Legendre transform to this HJB equation, the resulting equation (i.e., \eqref{HJB3} for the transformed value function $w(t,y)$) still contains a fully nonlinear term, induced precisely by $v(t,x,2)$ in the HJB equation. 

The fully nonlinear parabolic equation \eqref{HJB3}, compared with those commonly seen in mathematical finance, has a distinctive feature: a {\it mismatch} between the involved nonlinearlity and the terminal condition. As detailed at the beginning of Section~\ref{sec:dual}, if we take advantage of the specific structure of the equation (in the $y$ variable particularly), we can effectively simplify the nonlinearity and obtain a closed-form solution, which, however, violates the terminal condition; if we instead maintain the terminal condition in the first place, there is then no easy way to directly simplify the equation. To circumvent this, we ultimately construct a solution $w$ to \eqref{HJB3} from scratch, relying on fixed-point and compactness arguments. We first consider a regularized version of \eqref{HJB3}, whose solutions can be expressed as fixed points of an integral operator, defined on the space of continuous functions from $[0, T]$ to a suitable weighted Sobolev space. By showing that this operator is a contraction mapping, we obtain a solution to the regularized equation. As regularization diminishes, careful estimates ensure that the solution to the regularized equation converges to a solution $w$ to \eqref{HJB3}, on the strength of the Aubin--Lions
compactness theorem, and $w$ inherits many desirable properties (e.g., monotonicity and growth conditions) from the regularized solutions; see Theorem~\ref{th.for.w} for details. 

Interestingly, this solution $w$ to \eqref{HJB3} has a stochastic control characterization: it coincides with the value function of an optimal consumption problem with distinct intermediate and terminal utility functions, which are both isoelastic but have different exponents $q_1,q_2<0$; see \eqref{J^c} for the specifics. It is known in the literature that a negative exponent in an isoelastic utility function poses technical challenges to the standard verification argument. To overcome this, some employ an ``$\eps$-modification'' to the wealth dynamics (see e.g., Jane\v{c}ek and S\^{i}rbu \cite{JS12}, Guasoni and Huang \cite{GH19}), which however works for only additive or deterministic wealth processes (while our dual wealth process $Y$, given by \eqref{Y}, is stochastic and multiplicative); some others choose to focus on a smaller, more restrictive set of control strategies (see e.g., Aurand and Huang \cite{AH23}) for which the technical challenges can be avoided. In our case, we first choose to work with a smaller set of consumption strategies, i.e., those that are bounded, which allows us to perform the verification; next, by a detailed analysis of our dual wealth process, we show that the expected utility induced by any general (unbounded) consumption strategy can be achieved by a sequence of bounded ones, so that our verification result in fact extends to all general consumption strategies. This then proves the stochastic control characterization of $w$ (see Theorem~\ref{thm:convexity of tw}), which crucially implies the strict convexity of $y\mapsto w(t,y)$ (see Corollary~\ref{coro:convexity}).   

Thanks to this convexity, we can rigorously define the Fenchel--Legendre transform $\tv(t,x)$ of $w(t,y)$. By translating the properties of $w(t,y)$ in Theorem~\ref{th.for.w} into those of $\tv(t,x)$, we find that $\tv(t,x)$ is a classical solution to the {\it pre-event} HJB equation and has sublinear growth in $x$ (see Lemma~\ref{lem:tv}). Relying on this sublinear growth property,  we prove that the entire path of the optimal wealth process, when raised to any power within $(0,1)$, is uniformly integrable (see Proposition~\ref{prop:UI}). Such uniform integrability allows us to rigorously perform a verification argument, which shows that $\tv(t,x)$ coincides with the {\it pre-event} optimal value function $v(t,x,1)$ and the optimal trading strategy is the concatenation of a pre-event strategy (expressed in terms of derivatives of $\tv(t,x)$) and a post-event strategy (which is a constant, i.e., the standard Merton ratio); see Theorem~\ref{th.for.u}, the main result of this paper. 

Let us point out that the way we model switching risk preferences share similar spirit with traditional regime-switching models, but there is a key difference. Traditional regime-switching models focus on the modeling of changing investment opportunities, and their goal is to investigate how investors should adjust their trading strategy according to the changing market conditions; see e.g., Zhou and Yin~\cite{ZY}, Sotomayor and Cadenillas~\cite{Cad}, Pirvu and Zhang~\cite{PZ0, PZ}, Kwak, Pirvu, and Zhang~\cite{KPZ}, and Yao, Zhang, and Zhou~\cite{YZY}, among others. This paper, by contrast, models the changing risk preferences independently from the market conditions. In fact, the market conditions do not change in our model---it is a major life event that fundamentally reshapes investor's' risk aversion, which consequently alters their trading even when market conditions stay unchanged.

The rest of the paper is organized as follows. Section~\ref{sec:model} presents our model and derives associated HJB equations and a related dual equation under the Fenchel--Legendre transform. Section~\ref{sec:dual} analyzes the dual equation in detail, showing that it admits a positive convex classical solution with suitable growth. Section~\ref{sec:main} proves that the convex conjugate of the solution in Section~\ref{sec:dual} is exactly the value function of our optimal investment problem and derives an optimal trading strategy. Appendix~\ref{sec:construct w} is devoted to a detailed construction underlying results in Section~\ref{sec:dual}.

\section{The Model}\label{sec:model}

Let $(\Omega,\mathcal{F},\{\mathcal{F}_t\}_{t\ge 0},\P)$ be a filtered probability space where the filtration $\{\mathcal{F}_t\}_{t\ge 0}$ fulfills the usual conditions. Suppose that a standard Brownian motion $W$ and a random variable $\theta$ with an exponential law 
\begin{equation}\label{theta}
\P(\theta>t) = e^{-\lambda t}\quad \forall t\ge 0,\qquad \hbox{for some $\lambda>0$}, 
\end{equation}
are both defined on this probability space and they are independent of each other. 

Consider a financial market with
a risky asset $S$ (called the stock) whose price dynamics is modeled by a geometric Brownian motion 
 $$
dS_t=S_t\left[\mu\,dt +\sigma\,dW_t\right],\quad t\ge 0,
$$
where $\mu, \sigma>0$ are given constants. The market also has a risk-free asset whose value grows at a risk-free rate $r\ge 0$. Without loss of generality, we will assume throughout the paper that $r=0$.

Let $T>0$ be a fixed investment horizon. For any current time $t\in[0,T]$ and current wealth $x>0$, an investor intends to decide the proportion of her wealth, denoted by $\zeta_u$, to be invested in the stock $S$ at every time $s\in [t,T]$. The resulting wealth process is then given by  
\begin{align}
\label{equ:wealth-one}
d X^\zeta_u = X^\zeta_u\zeta_u \big(  \mu\, du+\sigma\, dW_u\big),\ \ u\in [t,T],\quad X^\zeta_t =x. 
\end{align}
We say the process $\{\zeta_u\}_{u\in [t,T]}$ is admissible if it satisfies suitable integrability, as defined below. 

\begin{definition}\label{def:portfolio-proportions}
For any $t\in[0,T]$, we say a real-valued progressively measurable process $\{\zeta_u\}_{u\in [t,T]}$ is {\em $t$-admissible} if 

      $\int_t^T \zeta^2_u\, du<\infty$ $\P$\text{-a.s.} 

We denote by $\mathcal A(t)$ the set of all such processes. 
\end{definition}

\begin{remark}
For any $t\in[0,T]$ and $\zeta\in\AA(t)$, \eqref{equ:wealth-one} admits a unique strong solution, given by 
\begin{equation}%
\label{equ:wealth-two}
    \begin{split}
      X_s= x \exp\left( \int_t^s\Big(\mu\zeta_u-\frac{\sigma^2\zeta^2_u}{2}\Big)\, du+\sigma \int_t^s \zeta_u \, dW_u \right),\ \ s\in[t,T].
    \end{split}
\end{equation}
\end{remark}

A major life event (e.g., parenthood, bereavement, a health shock, exposure to violent crime, and experiencing a natural disaster, etc.) may occur to the investor, independently to the financial market, and the (random) occurrence time is represented by the random variable $\theta$ in \eqref{theta}. As motivated in Section~\ref{sec:intro}, we stipulate that the major life event, once it occurs, will shift the investor's risk aversion to a higher level. Specifically, given 
\begin{equation}\label{p_1,p_2}
p_1,p_2\in (0,1)\quad \hbox{satisfying}\quad p_1>p_2, 
\end{equation}
we assume that the investor uses the isoelastic utility function 
\[
U(x,i)=\frac{x^{p_{i}}}{p_{i}},\quad i=1,2, 
\]
to evaluate her terminal wealth $X_T$, where ``$i=1$'' represents the case of $\theta>T$ (i.e., the major life event does not occur by time $T$, whence the investor maintains the usual risk aversion level $1-p_1$ at time $T$) while ``$i=2$'' represents the case of $\theta\le T$ (i.e., the major life event occurs by time $T$, whence the investor has a higher risk aversion level $1-p_2$ at time $T$). 

To facilitate subsequent analysis, let us introduce a two-state continuous-time Markov chain $\{\epsilon_t\}_{t\ge 0}$ that satisfies: (i) $\epsilon_t$ may switch from the state ``$i=1$'' to the other state ``$i=2$'', and the change time is $\theta$ as in \eqref{theta}; (ii) once $\epsilon_t$ is at the state ``$i=2$'', it will stay there forever. In other words, $\{\epsilon_t\}_{t\ge 0}$ has a rate matrix given by 
\begin{equation}\label{rate matrix}
A=[a_{ij}]_{i,j=1,2}\quad  \hbox{with}\quad a_{11} =-\lambda,\ a_{12}=\lambda,\ \hbox{and}\ a_{21}=a_{22} = 0. 
\end{equation}

For any current time $t\in [0,T]$, current wealth $x>0$, and current state $i\in\{1,2\}$, the investor looks for a portfolio process $\zeta^*\in\AA(t)$ that maximizes her expected utility from terminal wealth, i.e., attains the value 
\begin{equation} \label{v}
v(t,x,i):=\sup_{\zeta\in\AA(t)}\mathbb{E}^{t,x,i}\left[U\left(X^{\zeta}_T,\epsilon_T\right)\right], 
\end{equation}
where $\E^{t,x,i}$ denotes an expectation conditioned on $X^\zeta_t=x$ and $\epsilon_t=i$.

\subsection{The Hamilton--Jacobi--Bellman Equations}
Fix any current time and wealth $(t,x)\in[0,T]\times[0,\infty)$. If $i=\epsilon_t=2$, we have $\epsilon\equiv 2$ on the interval $[t,T]$, thanks to \eqref{rate matrix}. Hence, by \eqref{v}, 
\begin{equation}\label{v_2}
v(t,x,2) = \sup_{\zeta\in\AA(t)}\mathbb{E}^{t,x,2}\left[U\left(X^{\zeta}_T,2\right)\right]  = \sup_{\zeta\in\AA(t)}\mathbb{E}^{t,x}\left[\frac{(X^{\zeta}_T)^{p_2}}{p_2}\right],
\end{equation}
which is the classical Merton problem of optimal investment over the finite time horizon $T$. It is well-known that $v(t,x,2)$ satisfies the Hamilton--Jacobi--Bellman equation
\begin{align}\label{HJB v_2}
v_{t}(t,x,2) + \sup_{z\in\R}&\left[\mu x z v_{x}(t,x,2)+\frac{1}{2}\sigma^{2} x^{2} z^{2} v_{xx}(t,x,2)\right]=0,\quad v(T,x,2)= \frac{x^{p_2}}{p_2}. 
\end{align}
As the maximizer of the supremum above is $z^*=-\frac{\mu v_x(t,x,2)}{\sigma^2 x v_{xx}(t,x,2)}$ (assuming concavity of $x\mapsto v(t,x,2)$), the above equation becomes 
\begin{equation}\label{HJB112}
v_{t}(t,x,{{2}})-\frac{\mu^{2}}{2\sigma^{2}}\frac{v_{x}^{2}(t,x,{{2}})}{v_{xx}(t,x,{{2}})}=0,\quad v(T,x,2)= \frac{x^{p_2}}{p_2},
\end{equation}
This admits a unique classical solution 
\begin{equation}\label{f}
v(t,x,{{2}})=f(t)\frac{x^{p_{2}}}{p_2}\quad \hbox{with}\quad f(t):=\exp\left(-\frac{\mu^{2}}{2\sigma^{2}}(T-t)\right)>0, 
\end{equation}
which indeed fulfills the concavity assumption. Hence, we have $z^*=-\frac{\mu v_x(t,x,2)}{\sigma^2 x v_{xx}(t,x,2)} = \frac{\mu}{(1-p_2)\sigma^2}$ and the standard verification argument implies that
\begin{equation}\label{zeta^*_2}
\zeta^{2,*}_u \equiv \frac{\mu}{(1-p_2)\sigma^2},\quad u\in [t,T], 
\end{equation}
is an optimal portfolio process for \eqref{v_2}. 

On the other hand, suppose that $i=\epsilon_t=1$. For any $\zeta^1\in\AA(t)$, consider $\zeta\in\AA(t)$ defined by $\zeta_u := \zeta^1_u 1_{\{u < \theta\}} + \zeta^{2,*}_u 1_{\{u \ge \theta\}}$, $u\in[t,T]$, with $\zeta^{2,*}$ as in \eqref{zeta^*_2}. A heuristic application of It\^{o}'s formula for Markov-modulated processes (see e.g., \cite[Theorem 2.3]{EK22}) to $v(u,X^\zeta_u,\epsilon_u)$ yields
\begin{align}
&\E^{t,x,1}[v(T, X^\zeta_T,\epsilon_T)] -  v(t,x,1)\notag\\
&= \E^{t,x,1}\bigg[\int_t^{T\wedge\theta} v_t(s,X^\zeta_s,1) + \mu\zeta_sX^\zeta_s v_x(s,X^\zeta_s,1) +\frac12 \sigma^2 (X^\zeta_s)^2 v_{xx}(s,X^\zeta_s,1)\zeta^2_s\, ds\bigg]\nonumber\\
&\hspace{0.3in}  +\E^{t,x,1}\left[\left(v(\theta,X^\zeta_\theta,{2}) - v(\theta,X^\zeta_\theta,1)\right) 1_{\{\theta\le T\}}\right]\notag\\ 
&\hspace{0.3in} +\E^{t,x,1}\bigg[\int_{T\wedge\theta}^{T} v_t(s,X^\zeta_s,{2}) + \mu\zeta_s X^\zeta_s v_x(s,X^\zeta_s,{2}) +\frac12 \sigma^2 (X^{\zeta}_s)^2 v_{xx}(s,X^\zeta_s, 2)\zeta^2_s\, ds\bigg]\notag\\
&= \E^{t,x,1}\bigg[\int_t^{T\wedge\theta} v_t(s,X^{\zeta^1}_s,1) + \mu\zeta^1_sX^{\zeta^1}_s v_x(s,X^{\zeta^1}_s,1) +\frac12 \sigma^2 (X^{\zeta^1}_s)^2 v_{xx}(s,X^{\zeta^1}_s,1)(\zeta^1_s)^2\, ds\bigg]\nonumber\\
&\hspace{0.3in}  +\E^{t,x,1}\left[\left(v(\theta,X^{\zeta^1}_\theta,{2}) - v(\theta,X^{\zeta^1}_\theta,1)\right) 1_{\{\theta\le T\}}\right]\notag\\
&= \E^{t,x,1}\bigg[\int_t^{T} e^{-\lambda s} \bigg\{ v_t(s,X^{\zeta^1}_s,1) + \mu\zeta^1_s X^{\zeta^1}_s v_x(s,X^{\zeta^1}_s,1) +\frac12 \sigma^2 (X^{\zeta^1}_s)^2 v_{xx}(s,X^{\zeta^1}_s,1)(\zeta^1_s)^2 \notag\\
&\hspace{3.5in}+ \lambda \left(v(s,X^{\zeta^1}_s,2) - v(s,X^{\zeta^1}_s,1)\right) \bigg\}\, ds\bigg]
\label{Ito's'}
\end{align}
where the second equality follows from $\zeta_s = \zeta^{2,*}_s = \frac{\mu}{(1-p_2)\sigma^2}$ for $s\ge \theta$ as well as the HJB equation \eqref{HJB v_2} (where the supremum is achieved by 
$\frac{\mu}{(1-p_2)\sigma^2}$), and the third equality follows from the distribution of $\theta$ in \eqref{theta}. As $v(T, X^\zeta_T,\epsilon_T) = U(X^\zeta_T,\epsilon_T)$ by \eqref{v}, taking supremum over $\zeta$ (or equivalently $\zeta^1$) on the left-hand side of \eqref{Ito's'} gives $v(t,x,1)-v(t,x,1)=0$. This, along with \eqref{Ito's'}, suggests that the HJB equation for $v(t,x,1)$ be
\begin{align}\label{HJB v_1}
v_{t}(t,x,1) + \sup_{z\in \R}&\left[\mu z xv_{x}(t,x,1)+\frac{1}{2}\sigma^{2}z^{2}x^{2}v_{xx}(t,x,1)\right]-\lambda\big(v(t,x,1)-v(t,x,2)\big)=0,\quad 
\end{align}
with the terminal condition
\begin{equation}\label{TC v_1}
v(T,x,1)= \frac{x^{p_1}}{p_1}. 
\end{equation}
As the maximizer of the supremum in \eqref{HJB v_1} is $z^*=-\frac{\mu v_x(t,x,1)}{\sigma^2 x v_{xx}(t,x,1)}$ (assuming concavity of $x\mapsto v(t,x,1)$), \eqref{HJB v_1}-\eqref{TC v_1} reduce to 
\begin{equation}\label{HJB1}
v_{t}(t,x,1)-\frac{\mu^{2}}{2\sigma^{2}}\frac{v_{x}^{2}(t,x,1)}{v_{xx}(t,x,1)}-\lambda\big(v(t,x,1)-v(t,x,2)\big)=0,\quad v(T,x,1)=  \frac{x^{p_1}}{p_1}.
\end{equation}
As this equation depends on $v(t,x,2)$ given by \eqref{f}, it cannot be explicitly solved.

\subsection{The Dual Value Function}
Consider the Fenchel--Legendre transform (i.e., the convex dual) of $v(t,x,1)$, defined by
\begin{equation}\label{kj}
w(t,y) := \sup_{x>0}\{v(t,x,1)-xy\}\quad \forall (t,y)\in [0,T]\times (0,\infty).
\end{equation}
Heuristically, if we assume that $x\mapsto v(t,x,1)$ is concave and the first-order condition $v_x(t,x,1)-y=0$ has a solution $x^*(y)$ for all $y>0$, we will have $w(t,y) = v(t,x^*(y),1)-x^*(y)y$. Direct calculations then imply 
\[
w_t(t,y) = v_t(t,x^*(y),1),\quad w_y(t,y) = -x^*(y),\quad w_{yy}(t,y) = -(x^*)'(y) = - \frac{1}{v_{xx}(t,x^*(y),1)},
\]
where the last equality follows from differentiating $v_x(t,x^*(y),1)-y=0$ w.r.t.\ $y$ on both sides. It follows that the HJB equation \eqref{HJB1} for $v(t,x,1)$ turns into 
\begin{equation}\label{HJB3}
w_{t}(t,y)-\lambda w(t,y)+\lambda y w_{y}(t,y)+\xi y^2 w_{yy}(t,y)+\lambda f(t)\frac{(- w_{y}(t,y))^{p_2}}{p_2}=0,\quad w(T,y)=-\frac{y^{q_{1}}}{q_{1}},
\end{equation}
where we take 
\begin{equation}\label{q_1}
\xi := \frac{\mu^{2}}{2\sigma^{2}}>0,\qquad q_1<0\ \ \hbox{such that}\ \ \frac{1}{p_{1}}+\frac{1}{q_{1}}=1\ \hbox{(recall \eqref{p_1,p_2})},
\end{equation}
and recall $f(t)= \exp(-\xi (T-t))$ from \eqref{f}. It is worth noting that \eqref{HJB3} involves the nonlinear term $\lambda f(t)\frac{(- w_{y}(t,y))^{p_2}}{p_2}$ precisely because \eqref{HJB1} depends on $v(t,x,2)$ given by \eqref{f}.

\section{The Dual Problem}\label{sec:dual}
The goal of this section is to prove the existence of a classical solution $w$ to \eqref{HJB3}, show that it admits suitable properties (nonnegativity, growth conditions, and convexity), and establish a stochastic control characterization of it. 

To begin with, we observe that the challenge of finding a solution to \eqref{HJB3} stems from not only the involved nonlinearity (i.e., the term $\lambda f(t)\frac{(- w_{y}(t,y))^{p_2}}{p_2}$), but a {\it mismatch} between this nonlinear term and the terminal condition. Specifically, if we take 
\begin{equation}\label{q_2}
q_2<0\ \ \hbox{such that}\ \ \frac{1}{p_{2}}+\frac{1}{q_{2}}=1\ \hbox{(recall \eqref{p_1,p_2})}
\end{equation}
and consider the ansatz $w(t,y) = -u(t)^{q_2-1}\frac{y^{q_2}}{q_2}$, then the PDE in \eqref{HJB3} reduces to 
\[
u'(t) + (\lambda+\xi q_2) u(t) = \lambda f(t) u^2(t),
\]
which is a Bernouli equation that can be explicitly solved; see a related reduction to a Bernouli equation in \cite{GH19}. However, this closed-form solution $w(t,y) = -u(t)^{q_2-1}\frac{y^{q_2}}{q_2}$ by construction violates the terminal condition in \eqref{HJB3}. In other words, if we intend to simplify the nonlinearity by taking advantage of the specific PDE structure (in the $y$ variable particularly), we will end up violating the terminal condition. To maintain the terminal condition, there is then no easy way to simplify the PDE.

In view of this, we ultimately construct a solution $w$ to \eqref{HJB3} from scratch as follows. First, we consider a regularized version of \eqref{HJB3} (under appropriate change of variables) and express a solution to the regularized equation as a fixed point of an integral operator, defined on the space of continuous functions from $[0,T]$ to a weighted Sobolev space. By showing that this operator is a contraction mapping when the time horizon $T>0$ is small enough, we obtain local well-posedness of the regularized equation. A careful estimate then asserts that there is no blow-up in finite time, so that the local solution extends to any fixed $T>0$. Finally, by the Aubin--Lions compactness theorem, we are able to show that the global solution to the regularized equation converges to a solution to \eqref{HJB3}, as regularization diminishes. All the above construction is detailed in Appendix~\ref{sec:construct w}, which brings about the following result.

\begin{theorem}\label{th.for.w}
There is a solution $w\in C^{1+\gamma/2,2+\gamma}([0,T]\times (0,\infty))$ to \eqref{HJB3}, for some $\gamma\in(0,1)$. Moreover, $w$ satisfies the following properties: 
\begin{itemize}
\item [(i)] $w$ is strictly positive; specifically, 
\begin{equation}\label{inf tw}
w(t,y) > -\exp\big\{ (q_1-1)(\lambda + \xi q_1)(T-t)\big\}\frac{y^{q_{1}}}{q_{1}}>0. 
\end{equation}
\item [(ii)] $w$ is strictly decreasing in $y$.
\item [(ii)] There exist $\eta>0$ and $C>0$ (depending on $T>0$) such that
\begin{equation}\label{todo}
 w(t,y)\le C(1+y^{-\eta})\quad \forall (t,y)\in [0,T]\times (0,\infty). 
 \end{equation}
 \end{itemize}
\end{theorem}

\begin{proof}
See Section~\ref{subsec:proof of Thm tw}.
\end{proof}

\subsection{A Stochastic Control Characterization of $w$}
Recall $f$ in \eqref{f}. Let $\mathcal{C}$ be the set of $\R_+$-valued progressively measurable processes $c$ satisfying $\int_0^T f(s)c_s\, ds <\infty$ $\P$-a.s. For any $t\in[0,T]$, $y>0$, and $c\in\mathcal C$, let $Y^c$ be the unique solution to 
\begin{equation}\label{Y}
dY_s = \lambda\big(1 - f(s) c_s\big) Y_s\,  ds- \sqrt{2\xi} Y_s\, dW_s,\quad s\in[t,T],\quad Y_t=y. 
\end{equation}
We can then define the function $J^c:[0,T]\times(0,\infty)\to\R$ by
\begin{equation}\label{J^c}
J^c (t,y) := \E^{t,y}\left[\int_t^T  -e^{-\lambda(s-t)} f(s)\frac{(c_sY^c_s)^{q_2}}{q_2} - e^{-\lambda(T-t)} \frac{(Y^{c}_T)^{q_1}}{q_1}\right],
\end{equation}
where $\E^{t,y}$ denotes an expectation conditioned on $Y^c_t =y$ and $q_1,q_2<0$ are as in \eqref{q_1} and \eqref{q_2}. 

\begin{remark}\label{rem:J^c}
For any $c\in\mathcal C$, the unique solution to \eqref{Y} is given by 
\begin{equation}\label{exp form}
Y^c_s = y \exp\bigg\{ (\lambda-\xi) (t-s)- \lambda\int_t^s f(u) c_u\, du - \sqrt{2 \xi} (W_t-W_s)\bigg\}>0,\quad \forall s\ge t. 
\end{equation}
With $f(s)>0$, $c_s\ge 0$, and $Y_s>0$ for all $s\in [t,T]$, as well as $q_1,q_2<0$, we observe from \eqref{J^c} that $J^c(t,y) \ge 0$. Moreover, there exists some $c\in\mathcal C$ such that $J^c<\infty$. Indeed, for any $\ell>0$, by taking $c_s = \ell \exp\big(\xi(T-s)\big)$, we see that $f(s) c_s = \ell$ (so that $Y$ becomes a geometric Brownian motion) and $c$ is a positive process bounded away from 0. It can then be checked directly from \eqref{J^c} that $J^c (t,y)<\infty$ for all $(t,y)\in[0,T]\times(0,\infty)$. 
\end{remark}

In the following, we will show that the classical solution $w$ to \eqref{HJB3} obtained in Theorem~\ref{th.for.w} in fact equals $\inf_{c\in\mathcal C} J^c(t,y)$. To this end, we need the following two technical results. 

\begin{lemma}\label{lem:c<=c'}
For any $t\in[0,T)$ and $c,c'\in\mathcal C$ such that $c_s\le c'_s$ a.s.\ for all $s\in[t,T)$, we have $Y^c_s \ge Y^{c'}_s$ and $c_s Y^c_s \le c'_s Y^{c'}_s$ a.s.\ for all $s\in[t,T]$. 
\end{lemma}

\begin{proof}
As $c_s\le c'_s$ a.s.\ for all $s\in[t,T)$, 
\begin{equation}\label{exp>=exp}
\exp\bigg(-\lambda\int_t^u f(u) c_u\, du\bigg) \ge  \exp\bigg(-\lambda\int_t^u f(u) c'_u\, du\bigg),\quad \forall u\in [t,T]. 
\end{equation}
In view of \eqref{exp form}, this readily implies $Y^c_s \ge Y^{c'}_s$ a.s.\ for all $s\in[t,T]$. Now, observe that
\begin{align*}
    \exp\bigg(-\lambda\int_t^u f(u) c_u\, du\bigg) &= 1-\lambda\int_t^s f(u)c_u \exp\bigg(-\lambda\int_t^u f(\ell) c_\ell\, d\ell\bigg)\, du,\\
    \exp\bigg(-\lambda\int_t^u f(u) c'_u\, du\bigg) &= 1-\lambda\int_t^s f(u)c'_u \exp\bigg(-\lambda\int_t^u f(\ell) c'_\ell\, d\ell\bigg)\, du. 
\end{align*}
This, along with \eqref{exp>=exp}, entails
\[
\int_t^s f(u)\bigg\{c_u \exp\bigg(-\lambda\int_t^u f(\ell) c_\ell\, d\ell\bigg) - c'_u \exp\bigg(-\lambda\int_t^u f(\ell) c'_\ell\, d\ell\bigg)\bigg\}\, du\le 0,\quad\forall s\in[t,T],
\]
which in turn implies 
\[
c_u \exp\bigg(-\lambda\int_t^u f(\ell) c_\ell\, d\ell\bigg) \le c'_u \exp\bigg(-\lambda\int_t^u f(\ell) c'_\ell\, d\ell\bigg),\quad \hbox{for a.e.\ $u\in[t,T]$}. 
\]
Thanks to \eqref{exp form} again, this implies $c_s Y^c_s \le c'_sY^{c'}_s$ a.s.\ for all $s\in[t,T]$.
\end{proof}

If we take $c_s$ in \eqref{Y} to be of the feedback form $c(s,y) = \big(-w_y(s,y)\big)^{\frac{1}{q_2 -1}}/y$, where $w$ is the solution to \eqref{HJB3} obtained in Theorem~\ref{th.for.w}, then \eqref{Y} becomes
\begin{equation}\label{Y^*}
dY_s = \lambda \left(Y_s -  f(s) {\big(-w_y(s,Y_s)\big)}^{\frac{1}{q_2 -1}} \right)\,  ds-\sqrt{2\xi} Y_s\, dW_s,\quad s\in[t,T],\quad Y_t=y. 
\end{equation}

\begin{lemma}\label{lem:Y^*}
For any $(t,y)\in[0,T)\times(0,\infty)$, there exists a unique non-explosive strong solution $Y$ to \eqref{Y^*}. Moreover, $Y_s>0$ $\forall s\in[t,T]$ $\P$-a.s.
\end{lemma}

\begin{proof}
Define $b(s,y) := \lambda\big( y- f(s) {\big(-w_y(s,y)\big)}^{\frac{1}{q_2 -1}}\big)$ for $(s,y)\in [0,T]\times (0,\infty)$, such that \eqref{Y^*} can be written as 
\begin{equation}\label{Y^*'}
dY_s = b(s,Y_s)\, ds - \sqrt{2\xi} Y_s\, dW_s,\quad Y_t=y. 
\end{equation}
For each $n,m\in\N$, set $E_{n,m} := [1/n,m]$. As $w\in C^{1,2}([0,T]\times (0,\infty))$ and $w_y<0$ on $[0,T]\times (0,\infty)$, $y\mapsto b(s,y) := \lambda\big(y- f(s) {\big(-\tw_y(s,y)\big)}^{\frac{1}{q_2 -1}}\big)$ is Lipschitz and has linear growth on $E_{n,m}$, uniformly in $s\in[0,T]$. Hence, for any $y\in E_{n,m}$, there is a unique strong solution $Y^{(n,m)}$ to \eqref{Y^*'} up to a possibly finite explosion time $\tau_n\wedge \theta_m$, with
\[
\tau_n:= \inf\{s\in [t,T]: Y^{(n,m)}_s\le 1/n\}\quad \hbox{and}\quad \theta_m:=\inf\{s\in [t,T]: Y^{(n,m)}_s\ge m\}
\]
under the convention $\inf\emptyset = \infty$. It follows that for each fixed $n\in\N$,  
\[
Y^{(n)}_s := \sum_{m\in\N} Y^{(n,m)}_{s\wedge \tau_n} 1_{\{\theta_{m-1} \le s< \theta_m\}},\quad \hbox{with $\theta_0:=t$}, 
\]
is the unique strong solution to \eqref{Y^*'} up to a possibly finite explosion time $\tilde{\tau}_n\wedge \tilde{\theta}_\infty$, with
\[
 \tilde\tau_n:= \inf\{s\in [t,T]: Y^{(n)}_s\le 1/n\}\quad \hbox{and}\quad \tilde\theta_\infty := \lim_{m\to\infty}\inf\{s\in [t,T]: Y^{(n)}_s\ge m\}.
 \] 
under the convention $\inf\emptyset = \infty$. 
 
Set $\bar b(y):= \lambda y$ and consider the geometric Brownian motion $Y$ given by  
\begin{equation}\label{Y w/o c}
dY_s = \bar b(Y_s)\,  ds- \sqrt{2\xi} Y_s\, dW_s,\quad Y_t=y.
\end{equation}
As $\bar b(y) \ge b(s,y)$ on $[0,T]\times (0,\infty)$ by definition, by the comparison arguments in \cite[Proposition 5.2.18]{KS-book-91}, we get $Y_s\ge Y^{(n)}_s$ for all $s\in [t, \tilde\tau_n\wedge\tilde\theta_\infty]$. It follows that
\[
\tilde\tau_n\wedge \tilde \theta_\infty \ge \tilde\tau_n\wedge \lim_{m\to\infty}\inf\{s\in [t,T]: Y_s\ge m\} = \tilde\tau_n \wedge \infty = \tilde\tau_n,
\]
where the first equality follows from the fact $Y$, as a geometric Brownian motion, is non-explosive on the state space $(0,\infty)$. The above inequality indicates that if $Y^{(n)}$ explodes in finite time, it can only explode to the lower bound $1/n$, but not $+\infty$. Then, the process
\[
Y^*_s := \sum_{m\in\N} Y^{(n)}_{s} 1_{[\tilde\tau_{n-1}, \tilde\tau_n)}(s),\quad \hbox{with $\tilde\tau_0:=t$}, 
\]
is the unique strong solution to \eqref{Y^*} up to a possibly finite explosion time 
\[
\tau^*_\infty := \lim_{n\to\infty} \tau^*_n,\quad \hbox{with}\ \   \tau^*_n := \inf\{s\in [t,T]: Y^*_s\le 1/n\}.
\]
By applying It\^{o}'s formula to $s\mapsto e^{-\lambda(s-t)} w(s, Y^*_{s})$ up to $\tau^*_n$, we get 
\begin{align*}
&\E\left[e^{-\lambda(\tau^*_n-t)}  w(\tau^*_n, Y^*_{\tau^*_n})\right]- w(t,y)\notag\\
&=  \E\bigg[\int_t^{\tau^*_n} e^{-\lambda(s-t)} \bigg(w_t - \lambda w +  \lambda Y^*_s w_{y} + \lambda f(s) {\left(-w_y(s,Y^*_s)\right)}^{p_2}   + \xi (Y^*_s)^2 w_{yy}\bigg) (s, Y_s^*)\, ds \bigg]\\
&=  \E\bigg[\int_t^{\tau^*_n} e^{-\lambda(s-t)} \lambda f(s) \frac{{\left(-w_y(s,Y^*_s)\right)}^{p_2}}{q_2}\, ds\bigg]\le 0, \label{Ito's on w}
\end{align*}
where the second equality follows as $w$ satisfies \eqref{HJB3} and $\frac{1}{p_2}+\frac{1}{q_2}=1$, and the inequality results from $q_2<0$. Thanks to $\tau^*_n\le T$ and $w>0$, the above inequality implies 
\[
e^{\lambda(T-t)} w(t,y) \ge \E\left[  w(\tau^*_n, Y^*_{\tau^*_n})\right] \ge \P(\tau^*_n\le \ell)\cdot \inf_{s\in[0,T]}w(s,1/n),\quad \forall 0<\ell\le T,
\]
where the last inequality follows again from $w>0$. Hence, for any fixed $\ell>0$,
\[
\P(\tau^*_n\le \ell) \le \frac{e^{\lambda(T-t)} w(t,y)}{\inf_{s\in[0,T]}w(s,1/n)}\to 0,\quad\hbox{as}\ n\to\infty,
\]
where the convergence results from \eqref{inf tw}. It follows that $\P(\tau^*_\infty\le \ell) = \lim_{n\to\infty}\P(\tau^*_n\le \ell) =0$. This implies $\P(\tau^*_\infty=\infty)=1$, so that $Y^*$ is the unique non-explosive $(0,\infty)$-valued strong solution to \eqref{Y^*}. 
\end{proof}

Now, we are ready to present the stochastic control characterization of the solution $w$ to \eqref{HJB3}. 

\begin{theorem}\label{thm:convexity of tw}

The solution $w$ to \eqref{HJB3} obtained in Theorem~\ref{th.for.w} satisfies
\begin{equation}\label{bar w}
w(t,y) =  
\inf_{c\in \mathcal{C}} J^c(t,y),\quad \forall (t,x)\in[0,T]\times(0,\infty). 
\end{equation}

\end{theorem}

\begin{proof}

Fix $(t,y)\in[0,T)\times(0,\infty)$. The proof will be divided into three parts. 

{\bf Step 1:} Show that $w(t,y)\le \inf_{c\in \mathcal{C}_b} J^c(t,y)$, where $\mathcal C_b$ is the set of all $c\in\mathcal C$ that are bounded. Using the relation $\frac{1}{p_2}+\frac{1}{q_2}=1$, we observe that \eqref{HJB3} can be written as
\begin{equation}\label{HJB3'}
w_{t} - \lambda w + \lambda y w_{y} + \xi y^2 w_{yy} +\lambda f(t)\inf_{L\ge 0}\left\{(-w_y)L-\frac{L^{q_2}}{q_2}\right\} = 0,\qquad w(T,y)=-\frac{y^{q_{1}}}{q_{1}}.
\end{equation}
For any $c\in\mathcal{C}_b$, take $M>0$ such that $c_s\le M$ for all $s\in[0,T]$ and define $\tau_n:=\inf\{t\ge 0:Y^c_t\notin (1/n,n) \}$ for all $n\in\N$. Thanks to the smoothness of $w$, It\^{o}'s formula implies 
\begin{align}
&\E\left[e^{-\lambda(T\wedge\tau_n-t)}  w(T\wedge\tau_n, Y^c_{T\wedge\tau_n})\right]\notag\\
&= w(t,y)+ \E\bigg[\int_t^{T\wedge\tau_n} e^{-\lambda(s-t)} \bigg(w_t - \lambda w + \lambda\big(1- f(s)c_s\big) Y^c_s w_{y} + \xi (Y^c_s)^2 w_{yy}\bigg) (s, Y_s^c)\, ds \bigg]\notag\\
&\ge w(t,y)+ \E\left[\int_t^{T\wedge\tau_n} e^{-\lambda(s-t)} \lambda f(s) \frac{(c_s Y^c_s)^{q_2}}{q_2}\, ds\right],\label{Ito's on w}
\end{align}
where the inequality holds because $w$ is a solution to \eqref{HJB3}, and thus also to \eqref{HJB3'}. Let $\underline Y$ be a geometric Brownian motion given by 
\begin{equation}\label{Y_L}
d\underline Y_s = \lambda\big(1 - f^* M\big) \underline Y_s\,  dt - \sqrt{2\xi} \underline Y_s\, dW_s, 
\end{equation}
with $f^*:=\max_{t\in[0,T]} f(t)<\infty$. By the standard comparison theorem for one-dimensional SDEs, $Y^c_{s}\ge \underline Y_s$ for all $s\in [0,T]$. Hence, by \eqref{todo}, there exist $\eta,C>0$ such that
\begin{equation}\label{esti}
w(T\wedge\tau_n, Y^c_{T\wedge\tau_n})\le C(1+(Y^c_{T\wedge\tau_n})^{-\eta})\le C(1+(\underline Y_{T\wedge\tau_n})^{-\eta})\le C\bigg(1+ \max_{s\in [t,T]}\underline Y^{-\eta}_s \bigg),
\end{equation}
where the second inequality results from $Y^c_{s}\ge \underline Y_s$ for all $s\in [0,T]$. 
Note that $\underline Y^{-\lambda}$ remains a geometric Brownian motion, so that its running maximum $\max_{s\in[t,T]} \underline Y^{-\lambda}_s$ has finite expectation. Now, as the right-hand side of \eqref{esti} is independent of $n\in\N$ and has finite expectation, we conclude from the reverse Fatou lemma that 
\[
\limsup_{n\to\infty}\E\left[e^{-\lambda(T\wedge\tau_n-t)}  w(T\wedge\tau_n,Y^c_{T\wedge\tau_n})\right]\le \E\left[e^{-\lambda(T-t)}   w(T,Y^c_{T})\right] =  \E\left[-e^{-\lambda(T-t)} \frac{(Y^c_{T})^{q_1}}{q_1}\right].
\]
As $n\to\infty$ in \eqref{Ito's on w}, thanks to the above inequality and applying the monotone convergence theorem to the expectation on the right-hand side of \eqref{Ito's on w}, we obtain 
\begin{equation*}
\E^{t,y}\left[\int_t^T  -e^{-\lambda(s-t)} f(s)\frac{(c_sY^c_s)^{q_2}}{q_2}\, ds - e^{-\lambda(T-t)} \frac{(Y^{c}_T)^{q_1}}{q_1}\right] \ge w(t,y), 
\end{equation*}
i.e., $J^c (t,y)\ge w(t,y)$. Taking infimum over all $c\in\mathcal C_b$ yields $\inf_{c\in \mathcal{C}_b} J^c(t,y)\ge w(t,y)$. 

{\bf Step 2:} Show that $\inf_{c\in \mathcal{C}_b}J^c(t,y)=\inf_{c\in \mathcal{C}}J^c(t,y)$. 
For any $c\in\mathcal C$ such that $J^c(t,y)<\infty$ (such a $c$ exists thanks to Remark~\ref{rem:J^c}), define $c^M\in\mathcal C_b$ by $c^M_s := c_s\wedge M$, $s\in[0,T]$, for all $M\in\N$. By Lemma~\ref{lem:c<=c'} and \eqref{exp form}, $Y^{c^M}_s\downarrow Y^c_s$ and $c^M_s Y^{c^M}_s\uparrow c_s Y_s$ $\P$-a.s.\ as $M\to\infty$, for all $s\in[t,T]$. With $q_1<0$, we conclude from $Y^{c^M}_s\downarrow Y^c_s$ and the monotone convergence theorem that 
\begin{equation}\label{mono}
\lim_{M\to\infty}\E^{t,y}\left[- e^{-\lambda(T-t)} \frac{(Y^{c^M}_T)^{q_1}}{q_1}\right] = \E^{t,y}\left[- e^{-\lambda(T-t)} \frac{(Y^{c}_T)^{q_1}}{q_1}\right]. 
\end{equation}
On the other hand, with $q_2<0$ and $c^M_s Y^{c^M}_s$ increasing in $M$, we have
\begin{align}
0\le \int_t^T  -e^{-\lambda(s-t)} f(s)\frac{(c^M_sY^{c^M}_s)^{q_2}}{q_2}\, ds &\le \int_t^T  -e^{-\lambda(s-t)} f(s)\frac{(c^1_sY^{c^1}_s)^{q_2}}{q_2}\, ds,\quad \forall M>1. \label{upper bound}
\end{align}
We claim that the right-hand side above has finite expectation. As $c^1_s = c_s\wedge 1$, we see that
\[
c^1_sY^{c^1}_s \ge c_s Y^{c^1}_s 1_{\{c_s\le 1\}} + Y^{c^1}_s 1_{\{c_s> 1\}} \ge c_s Y^{c}_s 1_{\{c_s\le 1\}} + \underline Y_s 1_{\{c_s> 1\}},
\]
where the last inequality follows from $Y^{c^1}_s \ge Y^c_s$ and $Y^{c^1}_s \ge\underline Y_s$, where $\underline Y$ is the geometric Brownian motion \eqref{Y_L} with $M=1$. As $q_2<0$, this implies $(c^1_sY^{c^1}_s)^{q_2} \le (c_s Y^{c}_s)^{q_2} 1_{\{c_s\le 1\}} + (\underline Y_s)^{q_2} 1_{\{c_s> 1\}} \le (c_s Y^{c}_s)^{q_2}+(\underline Y_s)^{q_2}$. It follows that
\begin{align*}
\E^{t,y}\bigg[ \int_t^T  -e^{-\lambda(s-t)} f(s)\frac{(c^1_sY^{c^1}_s)^{q_2}}{q_2}\, ds \bigg]&\le \E^{t,y}\bigg[ \int_t^T  -e^{-\lambda(s-t)} f(s)\frac{(c_sY^{c}_s)^{q_2}}{q_2}\, ds \bigg]\\
&\hspace{0.2in}+ \E^{t,y}\bigg[ \int_t^T  -e^{-\lambda (s-t)} f(s)\frac{(\underline Y_s)^{q_2}}{q_2}\, ds \bigg]<\infty, 
\end{align*}
where the finiteness holds because (i) the first expectation on the right-hand side is bounded from above by $J^c(t,y)<\infty$ and (ii) the second expectation on the right-hand side is finite due to $\underline Y$ being a geometric Brownian motion. This, along with \eqref{upper bound}, ensures that we can use the dominated convergence theorem to conclude that
\[
\lim_{M\to\infty} \E^{t,y}\bigg[ \int_t^T  -e^{-\lambda(s-t)} f(s)\frac{(c^M_sY^{c^M}_s)^{q_2}}{q_2}\, ds \bigg] = \E^{t,y}\bigg[ \int_t^T  -e^{-\lambda(s-t)} f(s)\frac{(c_sY^{c}_s)^{q_2}}{q_2}\, ds \bigg].
\]
Combining this and \eqref{mono} amounts to $\lim_{M\to\infty} J^{c^M}(t,y)=J^c(t,y)$. This then implies $\inf_{c\in \mathcal C_b} J^c(t,y) = \inf_{c\in \mathcal C} J^c(t,y)$. 

{\bf Step 3:} Show that $w(t,y)\ge \inf_{c\in \mathcal C} J^c(t,y)$. 
Take $c\in\mathcal C$ in \eqref{Y} to be the candidate feedback optimal control $c^{*}(t,y) := {(-w_y(t,y))}^{\frac{1}{q_2 -1}}/y$. Then, \eqref{Y} becomes \eqref{Y^*}, which admits a unique strong solution $Y^*$ that is strictly positive (see Lemma~\ref{lem:Y^*}). With $(c^*,Y^*)$ in place of $(c,Y^c)$ in \eqref{Ito's on w}, the inequality therein becomes an equality (as $L = (-w_y(t,y))^{\frac{1}{q_2 -1}}$ attains the infimum in \eqref{HJB3'}) and we thus obtain 
\begin{align}
&\E\left[e^{-\lambda(T\wedge\tau_n-t)}  w(T\wedge\tau_n, Y^*_{T\wedge\tau_n})\right]\notag\\
&= w(t,y)+ \E\bigg[\int_t^{T\wedge\tau_n} e^{-\lambda(s-t)} \bigg(w_t - \lambda w + \lambda\big(1- f(s)c^*_s\big) Y^*_s w_{y} + a_3 (Y^*_s)^2 w_{yy}\bigg) (s, Y_s^*)\, ds \bigg]\notag\\
&= w(t,y)+ \E\left[\int_t^{T\wedge\tau_n} e^{-\lambda(s-t)} a_1 f(s) \frac{(c^*_sY^*_s)^{q_2}}{q_2}\, ds\right],\label{Ito's on w'}
\end{align}
By applying Fatou's lemma to the expectation on the left-hand side (thanks to $w>0$) and the monotone convergence theorem to the expectation on the right-hand side, we get
\begin{equation*}
\E^{t,y}\left[\int_t^T  -e^{-\lambda(s-t)} f(s)\frac{(c^*_sY^*_s)^{q_2}}{q_2}\, ds - e^{-\lambda(T-t)} \frac{(Y^*_T)^{q_1}}{q_1}\right] \le w(t,y), 
\end{equation*}
which implies $\inf_{c\in\mathcal C} J^c(t,y)\le w(t,y)$.

Combining results in the above three steps gives $w(t,y)=\inf_{c\in\mathcal C} J^c(t,y)$.  
\end{proof}

\begin{corollary}\label{coro:convexity}
    The solution $w$ to \eqref{HJB3} obtained in Theorem~\ref{th.for.w} is strictly convex in $y$. 
\end{corollary}

\begin{proof}
Let us prove $w(t,y) = \bar{w}(t,y):=\inf_{c\in\mathcal C} J^c(t,y)$ is convex in $y.$ Indeed let $y_1, y_2$ positive
 and
 $$ dY^i_t = (Y^i_t a_1 - p_2 f(t) C^i_t)dt+ Y^i_t \sqrt{2a_3} dW_t, Y^i_0=y_i, C^i\in \tilde{\mathcal{C}}, i=1,2,
  $$
where $\tilde{\mathcal{C}}$  is the set of controls so corresponding to $\mathcal{C}$ ($c\in \mathcal{C}$ iff $C=cY\in \tilde{\mathcal{C}}$)
 Then
 $$\E\left[\int_t^T  -e^{-a_1(s-t)}p_2 f(s)\frac{\left(   \frac{C^{1}_T + C^{2}_T}{2}  \right)^{q_2}}{q_2} - e^{-a_1(T-t)}  \frac{\left( \frac{Y^{1}_T + Y^{2}_T}{2} \right)^{q_1} }{q_1}\ \middle|\  \frac{Y^{1}_t+ Y^{2}_t}{2} =  \frac{y_1 + y_2}{2}   \right]$$
 $$ \leq\frac{1}{2}\E\left[\int_t^T  -e^{-a_1(s-t)}p_2 f(s)\frac{(C^{1}_s)^{q_2}}{q_2} - e^{-a_1(T-t)} \frac{(Y^{1}_T)^{q_1}}{q_1}\ \middle|\ Y^1_t =y_1 \right]  $$
 $$\quad +\frac{1}{2}\E\left[\int_t^T  -e^{-a_1(s-t)}p_2 f(s)\frac{(C^{2}_s)^{q_2}}{q_2} - e^{-a_1(T-t)} \frac{(Y^{2}_T)^{q_1}}{q_1}\ \middle|\ Y^2_t =y_2 \right]  $$
 by the concavity of functions $\frac{x^{q_1}}{q_1}$ and $\frac{x^{q_2}}{q_2},$ ({ since $q_1, q_2$ are negative}). Thus
 by taking infimum over $C^i\in \tilde{\mathcal C}, i=1,2,$
 $$\inf_{C^1,C^2\in \tilde{\mathcal C}} \E\left[\int_t^T  -e^{-a_1(s-t)}p_2 f(s)\frac{\left(   \frac{C^{1}_T + C^{2}_T}{2}  \right)^{q_2}}{q_2} - e^{-a_1(T-t)}  \frac{\left( \frac{Y^{1}_T + Y^{2}_T}{2} \right)^{q_1} }{q_1}\ \middle|\  \frac{Y^{1}_t+ Y^{2}_t}{2} =  \frac{y_1 + y_2}{2}   \right]$$$$\leq\ \frac{1}{2} [\bar{w}(t,y_1)+\bar{w}(t,y_2)],$$
 or
 $$ \bar{w}\left(t, \frac{y_1 + y_2}{2}\right) \leq\ \frac{1}{2} [\bar{w}(t,y_1)+\bar{w}(t,y_2)] $$
whence the claim.
\end{proof}

\section{The Main Result}\label{sec:main}
Consider the solution $w$ to \eqref{HJB3} obtained in Theorem~\ref{th.for.w} and define
\begin{equation}\label{tv}
\tv(t,x):= \inf_{y> 0}\{w(t,y)+xy\}\ge 0,\quad \forall (t,x)\in[0,T]\times (0,\infty).
\end{equation}
By the smoothness and strict convexity of $y\mapsto w(t,y)$ (Theorem~\ref{thm:convexity of tw} and Corollary~\ref{coro:convexity}), the inverse function of $y\mapsto\tw_y(t,y)$, denoted by $x\mapsto H(t,x)$, is well-defined, differentiable, and such that 
\begin{equation}\label{H}
    w_y(t,H(t,x))=-x,\quad \tv(t,x) = w(t,H(t,x))+x H(t,x),\quad \forall (t,x)\in[0,T]\times(0,\infty). 
\end{equation}
By direct calculation,
\begin{align}
  &\tv_t(t,x) = w_t(t,H(t,x)),\quad \tv_x(t,x) = H(t,x),\label{tw to tv}\\
  &\tv_{xx} (t,x)= H_x(t,x) = \frac{-1}{w_{yy}(t,H(t,x))}< 0,\label{tw to tv'} 
\end{align}
where the inequality follows again from the strict convexity of $w$. 

\begin{lemma}\label{lem:tv}
    $\tv$ defined in \eqref{tv} is a classical solution to the HJB equation \eqref{HJB v_1}. Moreover, there exists some $\widetilde C>0$ such that
    \begin{equation}\label{estimate}
f(t)\frac{x^{p_{1}}}{p_{1}}
 \leq \tv(t,x)\leq  \widetilde{C} (1+ x^{\frac{\eta}{\eta+1}}),
 \end{equation}
 where $f(t)$ is given by \eqref{f} and $\eta>0$ is taken from \eqref{todo}.
\end{lemma}

\begin{proof}

First, observe that
\begin{align}
&\tv_{t}(t,x) + \sup_{z\in\R}\left[\mu z x\tv_{x}(t,x)+\frac{1}{2}\sigma^{2} z^{2}x^{2}\tv_{xx}(t,x)\right]-\lambda[\tv(t,x)-f(t) x^{p_2}/p_2]\\
&=\tv_t(t,x) - \lambda[\tv(t,x)-f(t) x^{p_2}/p_2] - \xi \frac{(\tv_x(t,x))^2}{\tv_{xx}(t,x)}\nonumber\\
&= w_t(t,H(t,x)) - \lambda w(t,H(t,x))+ \lambda H(t,x) w_y(t,H(t,x))\nonumber\\
&\hspace{1in}+\xi H^2(t,x) w_{yy}(t,H(t,x))+\lambda f(t)[-w_y(t,H(t,x))]^{p_2}/p_2=0, \label{pde for tv}
\end{align}
where the first equality follows from $\tv_{xx}<0$ (recall \eqref{tw to tv'}), the second equality is due to the joint application of \eqref{H}, \eqref{tw to tv}, and \eqref{tw to tv'}, and the last equality holds because $w$ is a solution to \eqref{HJB3} (Theorem~\ref{th.for.w}).  
Also, by definition,
\begin{equation}\label{bc for tv}
\tv(T,x) = \inf_{y\ge 0}\{w(T,y)+xy\}=\inf_{y\ge 0}\{-y^{q_1}/q_1+xy\}=\frac{q_1-1}{q_1}x^{\frac{q_1}{q_1-1}}=x^{p_1}/p_1.
\end{equation}
We thus conclude that $\tv$ is a classical solution to \eqref{HJB v_1}. Next, recall the lower and upper bounds of $w(t,y)$ in \eqref{inf tw} and \eqref{todo}. Plugging them into \eqref{tv} directly yields \eqref{estimate}.
Theorem~\ref{th.for.w}
\end{proof}

\begin{proposition}\label{prop:UI}
Let $w$ be the solution $w$ to \eqref{HJB3} obtained in Theorem~\ref{th.for.w}. 
For any $t\in[0,T)$ and $y>0$, let $Y^*$ be the unique strong solution to \eqref{Y^*} and define
\begin{equation}\label{X^*}
X^*_s := -w_y(s,Y^*_s),\quad \hbox{for}\ s\in [t,T].    
\end{equation}
Then, 
\begin{equation}\label{hat zeta}
\zeta^*_s := -\frac{\mu \tv_x(s,X^*_s)}{\sigma^2 X^*_s \tv_{xx}(s,X^*_s)},\quad s\in [t,T],
\end{equation}
belongs to $\AA(t)$ and $X^*$ is the unique strong solution to \eqref{equ:wealth-one} with $x = -\tw_y(t,y)>0$ and $\zeta = \hat\zeta$ therein. Moreover, for any $\ell\in (0,1)$, 
\begin{equation}\label{E[X^eta]<infty}
\E^{t,x}[(X^*_{\tau\wedge T})^\ell] \le x^\ell \exp\bigg(\frac{\mu^2\ell}{2\sigma^2(1-\ell)}T\bigg)\quad \forall \hbox{stopping times}\ \tau,
\end{equation}
and $\{(X^*_{\tau\wedge T})^{\ell}:\tau\ \hbox{is a stopping time}\}$ is uniformly integrable. 
\end{proposition}

\begin{proof}
Recall from Theorem~\ref{th.for.w} that $w\in C^{1+\gamma/2,2+\gamma}([0,T]\times\R^d)$ for some $\gamma\in(0,1)$. As the term $g(t,y) := \lambda f(t)\frac{(- w_{y}(t,y))^{p_2}}{p_2}$ in \eqref{HJB3} satisfies $g_y\in C^{\gamma/2,\gamma}([0,T]\times\R^d)$, 
\cite[Theorem 10, p.72]{Friedman-book-64} implies that $w_{yyy}$ and $w_{ty}$ both exist and belong to $C^{\gamma/2,\gamma}([0,T]\times\R^d)$. Hence, by applying It\^{o}'s formula to $X^*$ in \eqref{X^*}, we get
\begin{align}
    dX^*_s &= -w_{yt}(s,Y^*_s)\, ds - w_{yy}(s,Y^*_s)\, dY^*_s-\frac12 w_{yyy}(s,Y^*_s) (dY^*_s)^2\notag\\
    &= \bigg[-w_{yt}(s,Y^*_s)-w_{yy}(s,Y^*_s) \left(\lambda Y_s - \lambda f(s) {\big(-w_y(s,Y_s)\big)}^{\frac{1}{q_2 -1}} \right) \notag\\
    &\hspace{1.5in}- \xi (Y^*_s)^2 w_{yyy}(s,Y^*_s)\bigg]\, ds + w_{yy}(s,Y^*_s)\sqrt{2\xi} Y^*_s\,  dW_s.\label{dX^*}
\end{align}
where the second line follows from \eqref{Y^*}. Moreover, $-w_{ty}$ can be computed from \eqref{HJB3} as 
\begin{align*}
    -w_{yt} = -w_{ty} = \lambda y w_{yy} + 2 \xi y w_{yy} + \xi y^2 w_{yyy}-\lambda f(t) (-w_y)^{p_2-1} w_{yy}.   
\end{align*}
Plugging this back into \eqref{dX^*} yields 
\begin{equation*}
    dX^*_s = 2 \xi Y^*_s w_{yy}(s,Y^*_s)\, ds + w_{yy}(s,Y^*_s)\sqrt{2\xi} Y^*_s\, dW_s. 
\end{equation*}
Thanks to \eqref{X^*}, we may apply \eqref{tw to tv} and \eqref{tw to tv'} (with $H(s,X^*_s) = Y^*_s$ therein) to the above dynamics and get 
\begin{align*}
    dX^*_s &= -2 \xi \frac{\tv_x(s,X^*_s)}{\tv_{xx}(s,X^*_s)}\, ds - \sqrt{2\xi} \frac{\tv_x(s,X^*_s)}{\tv_{xx}(s,X^*_s)}\, dW_s
    = X^*_s \left(\mu \hat\zeta_s\, ds + \sigma \hat\zeta_s\, dW_s\right), 
\end{align*}
where the last equality follows from $\xi = \frac{\mu^2}{2\sigma^2}$ (recall $\eqref{q_1}$), \eqref{hat zeta}, $X^*_s=-w_y(s,Y^*_s)$ being strictly positive. This readily shows that $X^*$ in \eqref{X^*} is a strong solution to \eqref{equ:wealth-one} (with $\zeta = \hat\zeta$ therein), so that $\hat\zeta$ is an admissible portfolio process (Definition~\ref{def:portfolio-proportions}). 

Fix $\ell\in(0,1)$. Thanks to \eqref{equ:wealth-two}, for any $s\ge t$, 
\begin{align*}
(X^*_s)^\ell &= x^{\ell} \exp\left\{ \int_t^s\Big(\ell\mu\zeta^*_u-\frac{\ell\sigma^2}{2}(\zeta^*_u)^2\Big)\, du+\int_t^s \ell\sigma\zeta^*_u \, dW(u) \right\}\\
&= x^{\ell} \exp\left\{ \int_t^s\Big(\ell\mu\zeta^*_u-\frac{\ell(1-\ell)\sigma^2}{2}(\zeta^*_u)^2\Big)\, du\right\}\\
&\hspace{1in} \cdot \exp\left\{ -\int_t^s\frac{\ell^2\sigma^2}{2}(\zeta^*_u)^2\, du +\int_t^s \ell\sigma\zeta^*_u \, dW(u)\right\}\\
&\le x^\ell \exp\bigg(\frac{\mu^2\ell T}{2\sigma^2(1-\ell)}\bigg) \exp\left\{ -\int_t^s\frac{\ell^2\sigma^2}{2}(\zeta^*_u)^2\, du +\int_t^s \ell\sigma\zeta^*_u \, dW(u)\right\},
\end{align*}
where the inequality holds because the maximum of the function $z\mapsto \mu z-\frac{(1-\ell)\sigma^2}{2}z^2$ on $\R$ is $\frac{\mu^2}{2\sigma^2(1-\ell)}$.
Note that the exponential term on the right-hand side above is a nonnegative local martingale (as it fulfills the dynamics $dZ_s = \ell\sigma\zeta^*_s Z_s\, dW_s$, $Z_t=1$), whence a supermartingale. As a result, for any stopping time $\tau$, by taking $s=\tau\wedge T$ in the above inequality and computing expectations on both sides, we get \eqref{E[X^eta]<infty}. Finally, as $\ell\in (0,1)$, there exists $p>1$ such that $\bar\ell := p\ell\in (\ell,1)$. Then, 
\[
\E\left[\left((X^*_{\tau\wedge T})^\ell\right)^p\right] = \E\left[(X^*_{\tau\wedge T})^{\bar\ell}\right]\le  x^{\bar\ell} \exp\bigg(\frac{\mu^2\bar \ell}{2\sigma^2(1-\bar \ell)}T\bigg),\quad \forall\hbox{ stopping time}\ \tau. 
\]
where the inequality follows from \eqref{E[X^eta]<infty}. With $p>1$ and the constant on the right-hand side above independent of $\tau$, we conclude that $\{(X^*_{\tau\wedge T})^{\ell}:\tau\ \hbox{is a stopping time}\}$ is uniformly integrable. 
\end{proof}

\begin{theorem}
For any $(t,x)\in[0,T]\times(0,\infty)$,
$
v(t,x,{1}) = \tv(t,x).
$

Moreover, 
\begin{equation}\label{optimal zeta}
   \hat \zeta_s := -\frac{\mu \tv_x(s,X^*_s)}{\sigma^2 X^*_s \tv_{xx}(s,X^*_s)} 1_{\{s<\theta\}} + \frac{\mu}{(1-p_2)\sigma^2}  1_{\{s\ge \theta\}},\quad s\in[t,T]
\end{equation}
is an optimal control for $v(t,x,1)$ (recall \eqref{v}).
\end{theorem}

\begin{proof}
For any $(t,x,i)\in[0,T]\times(0,\infty)\times\{{1},{2}\}$, define the function $\bar v(t,x,i)$ by $\bar v(t,x,{1}):=\tv(t,x)$ and $\bar v(t,x,{2}):=v(t,x,{2})=f(t)x^{p_2}/p_2$. Recall the continuous-time Markov chain $\epsilon_t$ given by \eqref{rate matrix}. At any initial time $t\in[0,T)$, take $\zeta\in\AA(t)$ and suppose that $X^\zeta_t =x$ and $\epsilon_t=1$. Consider $\tau_n:=\inf\{u\ge t : \sigma X^\zeta_u\bar v_x(u,X^\zeta_u,\epsilon_u)\zeta_u\ge n\}$ and recall $\theta$ in \eqref{theta}.
By applying It\^{o}'s formula for Markov-modulated processes (see e.g., \cite[Theorem 2.3]{EK22}) to $\bar v(t,X^\zeta_t,\epsilon_t)$, we have
\begin{align}
&\E[\bar v(T\wedge\tau_n, X^\zeta_{T\wedge\tau_n},\epsilon_{T\wedge\tau_n})]\notag\\
&= \tv(t,x) +\E\bigg[\int_t^{T\wedge\tau_n\wedge\theta} \tv_t(s,X_s) + \mu\zeta_sX^\zeta_s\tv_x(s,X^\zeta_s) +\frac12 \sigma^2 X_s^2 \tv_{xx}(s,X^\zeta_s)\zeta_s^2\, ds\bigg]\nonumber\\
&\hspace{0.3in}  +\E\left[\left(v(\theta,X^\zeta_\theta,{2}) - \tv(\theta,X^\zeta_\theta)\right) 1_{\{\theta\le T\wedge\tau_n\}}\right]\notag\\ 
&\hspace{0.3in} +\E\bigg[\int_{T\wedge\tau_n\wedge\theta}^{T\wedge\tau_n} v_t(s,X_s,{2}) + \mu\zeta_s X^\zeta_s v_x(s,X^\zeta_s,{2}) +\frac12 \sigma^2 X_s^2 v_{xx}(s,X^\zeta_s,{2})\zeta_s^2\, ds\bigg]\nonumber\\
& = \tv(t,x) +\E\bigg[\int_t^{T\wedge\tau_n} e^{-\lambda s} \bigg(\tv_t(s,X_s) + \mu\zeta_s X^\zeta_s \tv_x(s,X^\zeta_s) +\frac12 \sigma^2 X_s^2 \tv_{xx}(s,X^\zeta_s)\zeta_s^2 \notag\\
&\hspace{3.5in}+ \lambda \left(v(s,X^\zeta_s,{2}) - \tv(s,X^\zeta_s)\right) \bigg)\, ds\bigg]\notag\\
&\hspace{0.3in}+\E\bigg[\int_{T\wedge\tau_n\wedge\theta}^{T\wedge\tau_n} v_t(s,X_s,{2}) + \mu\zeta_s X^\zeta_s v_x(s,X^\zeta_s,{\bf 2}) +\frac12 \sigma^2 X_s^2 v_{xx}(s,X^\zeta_s,{2})\zeta_s^2\, ds\bigg]\nonumber\\
 &\le \tv(t,x),\label{Ito's}
\end{align}
where the second equality is due to \eqref{theta} and the inequality follows from \eqref{pde for tv} and \eqref{HJB v_2}. As $n\to\infty$, by Fatou's Lemma and \eqref{bc for tv}, we get $\E[U(X^\zeta_T,\epsilon_T)] = \E[\bar v(T, X^\zeta_T,\epsilon_T)]\le \tv(t,x)$. Then, taking supremum over all $\zeta\in\AA(t)$ yields $v(t,x,{1})\le \tv(t,x)$. 

On the other hand, by taking $\zeta=\hat\zeta$ in \eqref{optimal zeta}, as $-\frac{\mu \tv_x(s,x)}{\sigma^2 x \tv_{xx}(s,x)}$ attains the supremum in \eqref{pde for tv} and $\frac{\mu}{(1-p_2)\sigma^2}$ attains the supremum in \eqref{HJB v_2}, the inequality in \eqref{Ito's} becomes an equality, i.e.,
\begin{equation}\label{final for veri}
\E[\bar v(T\wedge\tau_n, X^{\hat \zeta}_{T\wedge\tau_n},\epsilon_{T\wedge\tau_n})]= \tv(t,x). 
\end{equation}
Note that by the definition of $\bar v(t,x,i)$, 
\begin{align}\label{bar v<=}
\bar v(T\wedge\tau_n, X^{\hat \zeta}_{T\wedge\tau_n},\epsilon_{T\wedge\tau_n}) &= \tv(T\wedge\tau_n, X^{*}_{T\wedge\tau_n}) 1_{\{\theta> T\wedge\tau_n\}} + f(T\wedge\tau_n)\frac{(X^{\hat \zeta}_{T\wedge\tau_n})^{p_2}}{p_2} 1_{\{\theta\le  T\wedge\tau_n\}}.
\end{align}
Observe that if $\theta\le  T\wedge\tau_n$, we have $X^{\hat \zeta}_{T\wedge\tau_n} = X^{*}_{\theta} G_{T\wedge\tau_n}$, where $G$ is a geometric Brownian motion given by
\[
dG_s = G_s \frac{\mu}{(1-p_2)\sigma^2} (\mu ds+\sigma dW_s),\quad s\in[\theta,T],\quad G_\theta=1.
\]
Thanks to this and the upper bound in \eqref{estimate}, we deduce from \eqref{bar v<=} that
\begin{equation}\label{bar v<='}
\bar v(T\wedge\tau_n, X^{\hat \zeta}_{T\wedge\tau_n},\epsilon_{T\wedge\tau_n}) \le \widetilde C(1+(X^*_{T\wedge\tau_n})^\eta)+ \frac{(X^{*}_{\theta})^{p_2} G^{p_2}_{T\wedge\tau_n}}{p_2}.
\end{equation}
By Proposition~\ref{prop:UI}, the first term on the right-hand side $\{\widetilde C(1+(X^*_{T\wedge\tau_n})^\eta)\}_{n\in\N}$ is uniformly integrable. As $p_2\in (0,1)$, we can take $b,b'>1$ such that $\bar{p_2}:= p_2 b\in (p_2,1)$ and $\bar{p_2}' := \bar{p_2} b'\in (\bar p_2,1)$. Consider $k,k'>1$ such that $\frac1b+\frac1k=1$ and $\frac{1}{b'}+\frac{1}{k'}=1$. Then, 
\[
\E\left[ \left((X^{*}_{\theta})^{p_2} G^{p_2}_{T\wedge\tau_n}\right)^b\right] \le \E\left[ (X^{*}_{\theta})^{\bar{p_2}} \max_{s\in[\theta,T]} G^{\bar p_2}_s\right] \le \left(\E\left[(X^{*}_{\theta})^{\bar{p_2}'}\right]\right)^{1/b'} \left(\E\left[\max_{s\in[\theta,T]} G^{\bar p_2 k'}_s\right]\right)^{1/k'}<\infty,
\]
where the second inequality follows from H\"{o}lder's inequality and the finiteness stems from \eqref{E[X^eta]<infty} and the fact that $G^{\bar p_2 k'}$ remains a geometric Brownian motion. This then implies that $\{(X^{*}_{\theta})^{p_2} G^{p_2}_{T\wedge\tau_n}\}_{n\in\N}$ is uniformly integrable. We thus conclude from \eqref{bar v<='} that $\{\bar v(T\wedge\tau_n, X^{\hat \zeta}_{T\wedge\tau_n},\epsilon_{T\wedge\tau_n})\}_{n\in\N}$ is uniformly integrable. Hence, as $n\to\infty$ in \eqref{final for veri}, we obtain $\E[\bar v(T,X^{\hat\zeta}_T, \epsilon_{T})] = \tv(t,x)$, i.e., $\E[ U(X^{\hat\zeta}_T, \epsilon_{T})] = \tv(t,x)$, which implies $v(t,x,1)\ge \tv(t,x)$. We therefore conclude $v(t,x,1) = \tv(t,x)$ and $\hat\zeta$ is an associated optimal control.  
\end{proof}

\appendix
\section{Construction of a Solution to \eqref{HJB3}}\label{sec:construct w}

This appendix is devoted to proving Theorem \ref{th.for.w}. The eventual proof is relegated to Section~\ref{subsec:proof of Thm tw}, which builds upon the developments in Sections~\ref{subsec:A.1}, \ref{subsec:A.2}, and \ref{subsec:A.3}.  

First, we perform the following change of variables
$$u(t,z)=\exp{( \beta t)} w(T-t, e^z), \ z\in \rr,$$
where
$$ \alpha := \frac{1}{2}\left[\frac{\lambda}{\xi}-1 \right]\in \Big(-\frac 12,\infty\Big)\quad\hbox{and}\quad \beta := \lambda+ \xi  \alpha^2 \geq 0. $$
By the chain rule, 
\[
u_t =\exp{( \beta t)} (\beta w- w_t),\quad u_z =\exp{( \beta t)} (\alpha w + y w_y),\quad u_z =\exp{( \beta t)} (\alpha w + y w_y). 
\]
By solving for $w_y$, one gets
\begin{equation}\label{eu}
w_y=\exp{(- \beta t)}u_z . 
\end{equation}
Therefore, the equation \eqref{HJB3} becomes
\begin{equation}\label{sys.v}
 \left\{
 \begin{array}{ll}
 u_t - \xi (\alpha^2u +2\alpha u_z+ u_{zz})=\exp(-\xi t)\exp((1-p_2\beta) t)\exp(-p_2z) (-u_z (t,z) )^{p_2},& z\in \rr,\\[10pt]
 	u(0,z)=-\frac{1}{q_{1}}\exp( q_1  z). &
 \end{array}
\right.
\end{equation}
The objective of this section is to prove the local/global well-posedness of solutions of this equation.

\subsection{Heat equation on weighted Sobolev spaces}\label{subsec:A.1}
For any $a>0$ and $1\leq p\leq \infty$, we consider the space
\[
  W^{1,p}_a(\rr)=\{ u:\rr\rightarrow \rr, \ s.t. \ e^{-a|x|}u, e^{-a|x|}u_x\in L^p(\rr)\}
\]
endowed with the norm
\[
  \|u\|_{W^{1,p}_a(\rr)}=\| e^{-a|x|}u\|_{L^p(\rr)}+\| e^{-a|x|}u_x\|_{L^p(\rr)}.
\]
We present some results for the action of the heat flow $\exp(-t\mathcal{L})$ associated with the operator 
\[\mathcal L u= \xi (\alpha^2u +2\alpha u_z+ u_{zz})\]
 on the weighted Sobolev spaces introduced above.

\begin{lemma}\label{Lp}
	There exists a positive constant $C$ such that for any $1\leq p\leq \infty$ the following
	\begin{equation}
\label{est.lp}
  \|e^{t\mathcal L} \varphi\|_{\lp}\leq Ce^{t\xi(a+|\alpha|)^2}\|\varphi\|_{\lp}, \forall\ t>0,
\end{equation}
holds for all $a>0$  and $\varphi\in \lp$.
	 \end{lemma}

\begin{proof}For simplicity let us assume that $\xi=1$. We denote by $g_t$ the one dimensional heat kernel $g_t(x)=(4\pi t)^{-1/2}\exp(-x^2/4t)$.
Solving explicitly the equation $u_t=\mathcal L u$ with initial data $u(0,z)=\varphi(z)$ we find that
\[
(e^{t\mathcal L} \varphi)(x)=(g_t(z)e^{-\alpha z}) \ast \varphi=\int_{\rr}g_t(x-y)e^{-\alpha(x-y)}\varphi(y)dt. 
\] 
	We consider the cases $p=1$ and $p=\infty$, the other cases being obtained by interpolation. Observe that we have
	\[
  (e^{t\mathcal L} \varphi)(x)=\int_{\rr} g_t(x-y)e^{-\alpha(x-y)}\varphi(y)dy=\int _{\rr} g_t(x-y)e^{a|y|}e^{-\alpha(x-y)}e^{-a|y|}\varphi(y)dy. 
\]
Thus
\[
  e^{-a|x|}|(e^{t\Delta}\varphi)(x)|\leq \| e^{-a|y|} \varphi(y)\|_{L^\infty(\rr)}  \int _{\rr} g_t(x-y)e^{-\alpha(x-y)}e^{-a|x|+a|y|}dy.
\]
Observe that for any $x\in \rr$ 
\begin{align*}
\label{}
   \int _{\rr} g_t(x-y)&e^{-\alpha(x-y)}e^{-a|x|+a|y|}dy\leq \int _{\rr} g_t(x-y)e^{(a+|\alpha|)|x-y|}dy\\
   &= \int _{\rr} g_t(z)e^{(a+|\alpha|)|z|}dz=\frac 2{\sqrt{4\pi t}}\int _0^\infty e^{-\frac {z^2}{4t}}e^{(a+|\alpha|)z}dz\\
   &=\frac 2{\sqrt{\pi }}\int _0^\infty e^{- {z^2}}e^{2(a+|\alpha|)z\sqrt t}dz=\frac {2e^{(a+|\alpha|)^2t}}{\sqrt{\pi }}\int _0^\infty e^{- (z-(a+|\alpha|)\sqrt t)^2}dz \\
  &\leq    Ce^{(a+|\alpha|)^2t} .
\end{align*}
This proves the case $p=\infty$. When $p=1$ we have
\[
  \int _\rr e^{-a|x|}| (e^{t\mathcal L}\varphi)(x)|dx\leq \| e^{-a|y|} \varphi(y)\|_{L^1(\rr)}\sup_{y}  \int _{\rr} g_t(x-y)e^{-\alpha(x-y)}e^{-a|x|+a|y|}dx.
\]
We obtain estimate \eqref{est.lp} with the same constant $C$ as in the first case. 
\end{proof}

\begin{lemma}\label{Wp}
	For any $a>0$ there exists a positive constant $C=C(a)$ such that   for any $1\leq p\leq \infty$  the following
	\begin{equation}
\label{est.grad.lp}
  \|\partial_x( e^{t\mathcal L} \varphi)\|_{\lp}\leq C e^{t\xi(a+|\alpha|)^2} \Big(1+\frac {1}{\sqrt t}\Big)\|\varphi\|_{\lp}, \forall\ t>0,
\end{equation}
holds for all $\varphi\in \lp$.
	 \end{lemma}

\begin{proof}
	Proceeding as in Lemma \ref{Lp} we have
	\[
  e^{-a|x|}|\partial_x(e^{t\mathcal L} \varphi)(x)|\leq  \| e^{-a|y|} \varphi(y)\|_{L^\infty(\rr)} \sup_{x}\int _{\rr} |\partial_x g_t(x-y)| e^{(a+|\alpha|)|x-y|}dy.
\]
Explicit computation show that
\begin{align*}
 e^{b^2t} \int _{\rr} |\partial_x g(z)|e^{b|z|}dz&=\frac 1{2t}\int _{\rr} g_t(z)|z|e^{b|z|}dz=\frac 1{\sqrt {\pi t}}\int_{\rr}|y|e^{-(y+b\sqrt t)^2}dy=\\
  &=\frac 1{\sqrt {\pi }}  \int_{\rr}\Big|\frac y{\sqrt t}-b\Big|e^{-y ^2}dy\leq C_b \Big(1+\frac {1}{\sqrt t}\Big).  
\end{align*} 
In a similar manner 
\[
  \int_{\rr} e^{-a|x|}|\partial_x(e^{t\mathcal L} \varphi)(x)|dx\leq \|e^{-a|x|} \varphi\|_{L^1(\rr)}   \int _{\rr} |(\partial_x g_t)(z)|e^{(a+|\alpha|)|z|}dz\leq  C_{a}{e^{(a+|\alpha|)^2t}}\Big(1+\frac 1{\sqrt t}\Big).
\]
and proof is finished.
\end{proof}

\subsection{Local well-posedness}\label{subsec:A.2}
We now consider the  system \eqref{sys.v}. We write it as 

\begin{equation}\label{sys1}
 \left\{
 \begin{array}{ll}
 u_t - \mathcal{L}u=h(t) \exp(-p_2z) |u_z (t,z) |^{p_2},& z\in \rr,\\[10pt]
 	u(0,z)=-\frac{1}{q_{1}}\exp( q_1  z). &
 \end{array}
\right.
\end{equation}
where $h(t)=\exp(-\xi t)\exp((1-p_2\beta) t)\in C([0,\infty))$.

\begin{theorem}
	\label{th.for.u} 	There exists a  solution of system \eqref{sys1} satisfying 
	\begin{equation}
\label{reg.u}
  u\in C([0,\infty)),\wi)
\end{equation}
for any $a\geq|q_1|.$ Moreover, it satisfies  
\begin{equation}
\label{prop.u}
  u(t,z)\geq -\frac{1}{q_1}\exp(t\xi(\alpha+q_1)^2)\exp(q_1z)>0, \, \forall \ t>0,\ \forall\, z\in \rr.
  \end{equation}
and
\begin{equation}\label{monotone.u}
  \partial_z u (t,z)\le 0.
\end{equation}
\end{theorem}

 In order to analyze the solutions of this system we consider a  perturbed system.
  For any $\eps\in (0,1)$ and $M>0$ we consider 

 \begin{equation}\label{sys.v.reg}
 \left\{
 \begin{array}{ll}
 u_t -   \mathcal{L}u=h(t)g_M(z) F_\eps(u_z),& z\in \rr,\ t>0,\\[10pt]
 	u(0,z)=u_0(z)=-\frac{1}{q_{1}}\exp{  (q_1   z)},& z\in \rr,
 \end{array}
\right.
\end{equation}
where \[
  F_\varepsilon(z)=(\varepsilon^2+z^2 )^{p_2/2}.
\]
and $g_M\in C^1(\rr)$ is a  bounded and decreasing approximation  of $\exp(-p_2z)$ (take for example a $C^1$ molifier of $\tilde g_M(z)=\min\{M,\exp(-p_2z)\}$)
\[
  |g_M(z)|\leq \tilde g_M(z), \ z\in \rr,
\]
satisfying $g_M(z)\rightarrow \exp(-p_2z)$ uniformly on compact sets.

\begin{theorem}\label{global.existence}
	Let $a\geq  |q_1| $. 
	
	i) For any $\eps, M>0$, there exists a unique global solution $u_{\epsilon,M}\in C([0,\infty)),\wi)$ to \eqref{sys.v.reg}. 
	
	ii) The solution satisfies  
		\begin{equation}
\label{est.sol}
  \|u_{\epsilon,M}(T)\|_{\wi}\leq C(T, \|h\|_{L^\infty((0,T))}, \|u_0\|_{\wi}), \quad \forall\  T>0,
\end{equation}
uniformly in $\eps\in (0,1)$ and $M>0$.

 Moreover,   the solution satisfies 
 \begin{equation}
\label{prop.u.eps}
  u_{\epsilon,M}\geq -\frac{1}{q_1}\exp(t\xi(\alpha+q_1)^2)\exp(q_1z)>0, \, \forall \ t>0,\ \forall\, z\in \rr
  \end{equation}
and
\begin{equation}\label{monotone}
  \partial_z u_{\epsilon,M}\le 0.
\end{equation}
\end{theorem}

\begin{proof}
	 Let us consider for simplicity the case $\xi=1$.
    Observe that $a\geq |q_1|$ guarantees that $u_0$ belongs to $\wi$.
	We consider the integral equation 
	\begin{equation}
\label{eq.integral}
  u(t)=\Phi_u(t)=e^{t\mathcal L}u_0 +\Psi_u(t)=e^{t\mathcal L}u_0 +\int _0^t e^{(t-s)\mathcal L}[ h(s) g_M F_\varepsilon(u_z )(s) ]ds.
\end{equation}
We prove that for suitable $A$ and $T$, $\Phi$ is a contraction in the space
\[
  X_{T,A}=\{u\in C([0,T],\wi), \|u\|_{X_{T,A}}=\max _{t\in [0,T]} \|u(t)\|_{\wi}\leq A\}.
\]
By Lemmas \ref{Lp} and \ref{Wp} we have
\begin{align*}
\label{}
  \|\Psi_u(t)-& \Psi_v(t)\|_{\wi} \\
  &\leq C\int _0^t  e^{(a+|\alpha|)^2(t-s)}(1+|t-s|^{-1/2}) \| h(s) g_M ( F_\varepsilon(u_z )(s)-  F_\varepsilon(v_z )(s)\|_{\li}ds\\
  &\leq C \|h\|_{L^\infty(0,T)} M  e^{(a+|\alpha|)^2T} \int _0^t (1+|t-s|^{-1/2})  \|F_\varepsilon(u_z )(s)-  F_\varepsilon(v_z )(s)\|_{\li}ds.
\end{align*}
Since $F_\varepsilon '(\theta)= p_2\theta (\varepsilon^2+\theta^2 )^{p_2/2-1}$, $p_2< 1$ we have $|F_\eps '(\theta)|\leq p_2\eps^{p_2-1}$ for all $\theta\in \rr$. It implies that
\[
   \|F_\varepsilon(u_z )(s)-  F_\varepsilon(v_z )(s)\|_{\li} \leq p_2\eps^{p_2-1} \|u-v\|_{\wi}
\]
and
\[
  \| \Phi_u -\Phi_v  \|_{X_{T,A}}\leq C(\eps )  e^{(a+|\alpha|)^2T}(T+\sqrt T) \|h\|_{L^\infty((0,T))} M \|u-v\|_{X_{T,A}}<\frac 12 \|u-v\|_{X_{T,A}},
\]
provided $T$ is small enough such that
\[
   C(\eps ) e^{(a+|\alpha|)^2T}(T+\sqrt T) \|h\|_{L^\infty((0,T))} M<\frac 12.
\]

Moreover, the same argument as above, taking $v=0$, shows that
\[
  \|\Phi_u\|_{X_{T,A}}\leq Ce^{(a+|\alpha|)^2  T}\|u_0\|_{\wi} +C(\eps) e^{(a+|\alpha|)^2T}(T+\sqrt T) \|h\|_{L^\infty((0,T))} M \|u\|_{X_{T,A}}.
\]
Choosing $A=4C\|u_0\|_{\wi}$, $T$ such that $e^{(a+|\alpha|)^2T}\leq 2$ and
\[
  C(\eps ) e^{(a+|\alpha|)^2T}(T+\sqrt T) \|h\|_{L^\infty((0,T))} M \leq \frac 12
\]
we obtain that $\Phi_u$ maps $X_{T,A}$ to itself. Applying Banach fix point theorem we obtain the existence of a unique local solution on the time interval $[0,T_0)$ with $T_0$ depending on $M$ and $\eps$. 

We now show that  the solutions are global. We exploit the fact that $F_\eps$ is a sublinear function. It satisfies 
 \begin{equation}\label{bound.F}
   |F_\eps(\theta)|\leq  C(p_2)(\eps ^{p_2}+|\theta|^{p_2})\leq  C(p_2)(1+|\theta|^{p_2}), \quad \forall \ \eps\in (0,1), \ \theta\in \rr.
\end{equation}
It implies that
\begin{align*}
\label{}
  \|\Psi_u(t)\|_{\wi}&\leq C\int _0^t  e^{(a+|\alpha|)^2(t-s)} (1+|t-s|^{-1/2}) \| h(s) g_M(z)  F_\varepsilon(u_z (s))\|_{\li}ds\\
  &\leq Ce^{(a+|\alpha|)^2T} M  \|h\|_{L^\infty((0,T))}  \int _0^t  (1+|t-s|^{-1/2})  \| 1+  |u_z (s)|^{p_2}\|_{\li}ds\\
  &\leq C\|h\|_{L^\infty((0,T))} e^{(a+|\alpha|)^2T}  M \sup _{z\in \rr} e^{-a|z|(1-p_2)}   \int _0^t  (1+|t-s|^{-1/2})  (1+ \|u(s)\|_{\wi}^{p_2} ).
\end{align*}
Since   $u$ is solution of the integral equation $u=\Phi_u$ we have for the maximal interval of existence 
 \[
   \|u(t)\|_{X_{T,A}}\leq e^{(a+|\alpha|)^2T}\Big[C\|u_0\|_{\wi}+ C \|h\|_{L^\infty((0,T))} M (T+\sqrt T) (1+\|u\|_{X_{T,A}}^{p_2})\Big].
\]
Since $p_2<1$ this excludes the blow-up in finite time.

Let us now show that for $a\geq |q_2|=\frac{p_2}{1-p_2}>p_2$
  the uniform estimate \eqref{est.sol} holds.
We  use  that  $|g_M(z)|\leq e^{-p_2z}$  (uniformly in $M>0$) and  the bound  \eqref{bound.F} holds uniformly in $\eps\in (0,1)$:
\begin{align*}
\label{}
 \|u(t)\|_{\wi}&= \|\Psi_u(t)\|_{\wi}\leq C_T\int _0^t    (1+|t-s|^{-1/2}) \| h(s) g_M(z) ( F_\varepsilon(u_z (s))\|_{\li}ds\\
  &\leq C_T   \|h\|_{L^\infty((0,T))}  \int _0^t  (1+|t-s|^{-1/2})  \|e^{-p_2z}(1+  |u_z (s)|^{p_2}\|_{\li}ds\\
  &\leq C_T \sup _{z\in \rr} e^{-a|z|(1-p_2)-p_2z}     \int _0^t  (1+|t-s|^{-1/2})  (1+ \|u(s)\|_{\wi}^{p_2} ).
\end{align*}
Again, since $p_2<1$ there is no blow-up in finite time and moreover \eqref{est.sol} holds.
 
Let us now prove property \eqref{prop.u.eps}. It is sufficient to observe that
\[
 {u}_*(t,z)=-\frac{1}{q_1}\exp(t\xi (\alpha+q_1)^2)\exp(q_1z)
\]
satisfies $  \partial_ tu_*-\mathcal Lu_*\leq 0\leq \partial_ tu_{\eps,M} -\mathcal Lu_{\eps,M}$. Since $e^{-a|z|}(u_{\eps,M}-u_*)(t)\in  L^\infty(\rr)$ we get $(u_{\eps,M}-u_*)(t,z)\gtrsim -\exp(-a|z|^2)$ and then by  the maximum principle \cite[Th.~9, p.~43]{Friedman-book-64} we obtain the desired property. 

Let us now prove property \eqref{monotone}.
Observe  that $\eta=u_z$ satisfies 
 \begin{equation}\label{sys.w.reg}
 \left\{
 \begin{array}{ll}
 \eta_t - \xi (\alpha^2\eta +2\alpha \eta_z+ \eta_{zz})=h(t) g_M'(z) F_\eps(\eta)+h(t)g_M(z) F_\eps'(\eta)\eta_z ,& z\in \rr,\\[10pt]
 	\eta(0,z)=-\exp( q_1  z). &
 \end{array}
\right.
\end{equation}
Since  $h(t)g_M(z) F_\eps'(\eta)$ is  a bounded function  and $h(t) g_M'(z) F_\eps(\eta) \leq 0$ a
we can apply the maximum principle \cite[Th.~9, p.~43]{Friedman-book-64} to obtain that $\eta(t,z)\leq 0$ for all $t>0$ and $z\in \rr$. 
The proof is now complete. 
\end{proof}

\subsection{Compactness and passage to the limit}\label{subsec:A.3}
Fix $R>0$ and set $I_R := (-R,R)$.
Through the section  $a > \frac {p_2}{1-p_2}, $  and   $u_{\varepsilon,M}$ solves \eqref{sys.v.reg}.  
We first prove  some estimates for the solutions, estimates which are uniform with respect to  the two parameters $\eps\in (0,1)$ and $M>0$.
\begin{lemma}
	\label{unif.H2}
	For any  $T$ and $R$ positive there exist a positive constant $C(T,R)$ such that the following estimates
	\begin{equation}
  \label{h2}
  \|u_{\eps,M}\|_{L^2((0,T),H^2(I_R))}\leq C(T,R),
\end{equation}
 	\begin{equation}
  \label{time.l2}
  \|\partial_t u_{\eps,M}\|_{L^2((0,T),L^2(I_R))}\leq C(T,R),
\end{equation}
hold uniformly for all $\eps>0$ and $M>0$. 
\end{lemma}

\begin{proof}[Proof of Theorem \ref{th.for.u}.]
	For any  $T$ and $R$ positive Lemma \ref{unif.H2} gives a positive constant $C(T,R)$ such that the following 
	\begin{equation}
  \label{h2.der.z}
  \|\partial_z u_{\eps,M}\|_{L^2((0,T),H^1(I_R))}+  \|\partial_{tz} u_{\eps,M}\|_{L^2((0,T),H^{-1}(I_R))}\leq C(T,R),
\end{equation}
holds uniformly for all $\eps\in (0,1)$ and $M>0$. 

All these estimates allow us to use Aubin-Lions compactness criterion both for $u_{\eps,M}$ and
$\partial _z u_{\eps,M}$. It implies that up to a subsequence, $u_{\eps_k,M_k}\rightarrow u$   and $\partial_z u_{\eps_k,M_k}\rightarrow \partial_x u$ $L^2_{loc}((0,\infty)\times \rr)$. Moreover, $u_{\eps_k,M_k}\rightarrow u$ and $\partial_z u_{\eps_k,M_k}\rightarrow \partial_x u$ a.e. in  $ (0,\infty)\times \rr$. All these estimates allows us to pass to the limit in the weak formulation of \eqref{sys.v.reg} to obtain a solution $u\in C([0,\infty),\wi)$ of system \eqref{sys1}. Moreover, properties
\eqref{prop.u.eps} and 
  \eqref{monotone} transfer from $u_{\eps,M}$ to $u$.
\end{proof}

\begin{proof}[Proof of Lemma \ref{unif.H2}]
	By Theorem \ref{global.existence} we have 
	\[
	 \|u_{\epsilon,M}(T)\|_{\wi}\leq C(T,   \|u_0\|_{\wi}), \quad \forall\  T>0.
	\]
	This implies the local estimate
	 \begin{equation}
\label{local.h1}
  \|u_{\eps,M}\|_{L^2((0,T),H^1(I_R))}\leq C(T,R).
\end{equation}
	In order to prove estimate \eqref{h2} it is sufficient to prove 
	\begin{equation}
\label{second.derivate}
  \|\partial_{xx}u_{\eps,M}\|_{L^2((0,T),L^2(I_R))}\leq C(T,R).
\end{equation}
Moreover, since $|g_M(z)|\leq \exp(-p_2z) $ and $F_\eps(z)\leq (1+|z|)^{p_1}\lesssim 1+|z|$
  we automatically obtain that 
\begin{align*}
  \|\partial_t u_{\eps,M}&\|_{L^2((0,T),L^2(I_R))}  \leq  \|\mathcal{L}u_{\eps,M}+h(t)g_M(z) F_\eps(\partial_z u_{\eps,M})\|_{L^2((0,T),L^2(I_R))} \\
  & \leq \|u_{\eps,M}\|_{L^2((0,T),H^2(I_R))}+\|h\|_{L^\infty(0,T)}\|\exp(-p_2z) (1+|\partial_z u_{\eps,M}|)\|_{L^2((0,T),L^2(I_R))}\\
  &\leq (1+\|h\|_{L^\infty(0,T)}\exp(p_2R)) \|u_{\eps,M}\|_{L^2((0,T),H^2(I_R))}\leq C(T,R).
\end{align*}

In the following we prove estimate \eqref{second.derivate}. For simplicity we will not make explicit the dependence of $u$ by $\eps$ and $M$. 
Let $\eta\in C_c^\infty(I_{2R}) $ such that $\eta\equiv 1$ in $I_R$.  Multiplying the equation \eqref{sys.v.reg} by $(\eta^2 u_z)_z$, denoting $f_\varepsilon(t,z)
:= h(t)\,g_M(z)\,F_\varepsilon(u_z)$ and using integration by parts we obtain 
\begin{align*}
\frac{1}2\frac {d}{dt}\int_\rr \eta^2 u_z^2(t,z) \, dz = &\int _{\rr} \eta^2 u_{z}u_{zt}\, dz= -\int _{\rr} (\eta^2 u_{z})_zu_{t}\,dz   \\
=  &  -\int _{\rr} (\eta^2 u_{z})_z [f_\eps+\xi(u_{zz}+2\alpha u_z+\alpha^2 u) ]\, dz.
\end{align*}
Observe that 
\begin{align*}
I := \xi\int_{\rr} u_{zz} (\eta^2 u_z)_z\,dz
&=
\xi\int_{\rr} \eta^2 u_{zz}^2\,dz
+2\xi\int_{\rr} u_z u_{zz}\eta\eta_z\,dz\\
&\geq \xi\int_{\rr} \eta^2 u_{zz}^2\,dz
- \frac {\xi} 2\int_{\rr} \eta^2 u_{zz}^2\,dz
-2\xi
\int_{\rr} u_z^2 \eta_z|^2\,dz \\
&\ge
\frac {\xi}2
\int_{\rr} \eta^2 u_{zz}^2\,dz
-
2 \xi\int_{\rr} u_z^2 \eta_z^2\,dz .
\end{align*}
It follows that 
\[
\frac{1}2\frac {d}{dt}\int_\rr \eta^2 u_z^2(t,z) +\frac {\xi}2
\int_{\rr} \eta^2 u_{zz}^2\,dz\leq 2 \xi\int_{\rr} u_z^2 \eta_z^2\,dz-\int _{\rr} (\eta^2 u_{z})_z [f_\eps +\xi( 2\alpha u_z+\alpha^2 u) ]\, dz.
\]
In the last two terms we integrate by parts to get
\[
-\int _{\rr} (\eta^2 u_{z})_z ( 2\alpha u_z+\alpha^2 u) ]\, dz=2\alpha \int _{\rr} \eta^2 u_{z}u_{zz}+\alpha^2\int_\rr \eta^2u_z^2=\int _{\rr}u_z^2 (-2\alpha \eta \eta_z+\alpha^2 \eta^2)
\]
and
\[
\frac{1}2\frac {d}{dt}\int_\rr \eta^2 u_z^2(t,z) +\frac {\xi}2
\int_{\rr} \eta^2 u_{zz}^2\,dz\leq   \xi\int_{\rr} u_z^2 (2\eta_z^2-2\alpha \eta \eta_z+\alpha^2 \eta^2) \,dz-\int _{\rr} (\eta^2 u_{z})_z  f_\eps  \, dz.
\]
Finally we estimate the term involving $f_\eps$
\begin{align*}
\Big|  \int_{\rr} f_\varepsilon (\eta^2 u_z)_z\,dz\Big| &
=
\Big| \int_{\rr} f_\varepsilon \eta^2 u_{zz}\,dz
+2\int_{\rr} f_\varepsilon \eta\eta_z u_z\,dz\Big| \\
&\le
\frac{\xi}{4}\int_{\rr} \eta^2 u_{zz}^2\,dz
+ \frac{1}{\xi}
\int_{\rr} \eta^2 f_\varepsilon ^2\,dz +   \int_{\rr}  f_\varepsilon ^2 \eta^2\,dz
+  \int_{\rr}  u_z ^2  \eta_z ^2\,dz .
\end{align*}
Combining the above estimates and using that 
\[f_\varepsilon(t,z)
:= h(t)\,g_M(z)\,F_\varepsilon(u_z)\leq h(t) \exp(-p_2z)(1+|u_z|)\]
we obtain for any $0<t<T$ that
\begin{align*}
\frac{1}2\frac {d}{dt}\int_\rr &\eta^2 u_z^2(t,z) +\frac {\xi}4
\int_{\rr} \eta^2 u_{zz}^2\,dz \lesssim \int_{\rr} u_z^2 ( \eta_z^2 +\eta^2) \,dz+\int_{\rr} |f_\varepsilon|^2 \eta^2\,dz. \\
&\lesssim  \int_{\rr} u_z^2 ( \eta_z^2 +\eta^2) \,dz +C(R,T)\int_{\rr} (1+u_z^2) \eta^2\,dz.
\end{align*}
Integrating over $t\in(0,T)$ yields
\begin{align*}
\frac{1}{2}\int_{I_R} (u_z(z,T)^2- u_z(z,0)^2 )\eta^2 \,dz &+ \frac{\xi}{4} \int_0^T\!\!\int_{I_R}
   u_{zz} ^2 \eta^2\,dz\,dt\\
 &\leq \int_0^T\int_{\rr} u_z^2 ( \eta_z^2 +\eta^2) \,dz +C(R,T)\int_0^T\int_{\rr} (1+u_z^2) \eta^2\,dz.
 \end{align*}
This gives
\begin{align*}
\int_0^T\!\!\int_{\rr}  u_{zz} ^2 \eta^2\,dz\,dt
&\lesssim 
\int_{\rr}  u_z(z,0) ^2 \eta^2 \,dz
+
C(R,T)\int_0^T\!\!\int_{I_{2R}} (1+ u_z^2) \, dz.
\end{align*}
Using that  $\eta\equiv 1$ in $I_R$ and the uniform local estimate \eqref{local.h1} we obtain the desired estimate \eqref{second.derivate} for  the second derivative of $u_{\eps,M}$.
\end{proof}

 \subsection{Proof of Theorem \ref{th.for.w} } \label{subsec:proof of Thm tw}
  The existence follows from Theorem \ref{th.for.u}. Classical parabolic Schauder estimates show that $u\in C^{1+\gamma/2,2+\gamma}([0,\infty)\times \rr)$ for some $\gamma>0$. 
  The properties of $w$ are obtained from the ones on $u$ in Theorem \ref{th.for.u}. 
To prove positivity of ${w}$, we recall the change of variables
\[
w(t,y)=e^{ - \beta (T-t)}u(T-t,\log y),
\]

and using \eqref{prop.u} we  conclude

\begin{equation}\label{es1}
w(t,y) \geq w^*(t,y)=-\exp\left( (q_1-1)(\lambda+ \xi q_1)(T-t)\right)\frac{y^{q_{1}}}{q_{1}}>0,
 \end{equation}
which establishes \eqref{inf tw}. Let $a\geq|q_1|.$

By Theorem \ref{th.for.u}  
we obtain
\[
0<u(t,z)|\le C_T e^{a|z|}\qquad\text{for a.e. } (t,z) \in [0,T]\times \mathbb{R}.
\]

Rewriting the change of variables gives, for $y>0$ with $z=\log y$,
$$
w(t,y)=e^{ - \beta (T-t)}\,u(T-t,\log y)\leq C_T e^{a |\log y|}.
$$
If $0<y\le 1$, then $|\log y|=-\log y$ and
$e^{a|\log y|}=y^{-a}$, so $w(t,y) \le C_T y^{-a}.$

If $y \ge 1$, as ${w}\ge 0$ and ${w}$ is decreasing in $y$, 
$
0 \le  w(t,y) \le  w(t,1) \le \sup_{s\in[0,T]}  w(s,1)=: C_T'.
$
Combining the two cases gives
\begin{equation}\label{es2}
\tw(t,y)
\le
(C_T\vee C_T')(1+y^{-a})
\qquad
\forall (t,y)\in[0,T]\times(0,\infty),
\end{equation}
which finishes the proof.

\section*{Acknowledgments}

Work supported by NSERC under research grant RGPIN-2019-05397.\\
L. I.  Ignat  was partially supported by a grant of the Ministry of Research, Innovation, and Digitization, CCCDI -
UEFISCDI, project number ROSUA-2024-0001, within PNCDI IV.
\\
Special thanks to Ivar Ekeland who formulated this research problem.

\end{document}